\documentclass{amsart}

\usepackage{amssymb,amsmath,xcolor,graphicx,xspace,colortbl,rotating} %
\usepackage[raggedrightboxes]{ragged2e} 
\usepackage{amsfonts}  
\usepackage{amsmath}  
\usepackage{amssymb}  
\usepackage{graphicx}  
\graphicspath{{modBVimrnHal17_graphics/}{modBVimrnHal17_tcache/}{modBVimrnHal17_gcache/}}
\DeclareGraphicsExtensions{.pdf,.eps,.ps,.png,.jpg,.jpeg}
\providecommand{\U}[1]{\protect\rule{.1in}{.1in}}
\newtheorem{theorem}{Theorem}
\theoremstyle{plain}

\newtheorem{definition}{Definition}

\newtheorem{proposition}{Proposition}

\numberwithin{equation}{section}
\begin{document}
\title[Modular operads and BV-geometry]{Modular operads and Batalin-Vilkovisky geometry\protect\footnote{
MPIM(Bonn) preprint 2006-48, (04/2006). IMRN, doi 10.1093/imrn/rnm075
}}
\author{Serguei
Barannikov}
\address{Ecole Normale Superieure, 45, rue d'Ulm 75230, Paris France }
\email{serguei.barannikov@ens.fr}
\begin{abstract}I describe the noncommutative Batalin-Vilkovisky geometry associated naturally with arbitrary modular operad. The classical limit of
this geometry is the noncommutative symplectic geometry of the corresponding tree-level cyclic operad. I show, in particular, that the algebras over the
Feynman transform of a twisted modular operad $\mathcal{P}$ are in one-to-one correspondence with solutions to quantum master equation of Batalin-Vilkovisky geometry
on the affine $\mathcal{P} -$manifolds. As an application I give a construction of characteristic classes with values in the homology
of the quotient of Deligne-Mumford moduli spaces. These classes are associated naturally with solutions to the quantum master equation on affine $\mathbb{S} [t] -$manifolds, where $\mathbb{S} [t]$ is the twisted modular $D e t -$operad constructed from symmetric groups, which generalizes the cyclic operad of associative algebras. 
\end{abstract}
\maketitle

\section{\bigskip Introduction.}
Modular operads, introduced in \cite{GK}, is a generalization
of cyclic operads intended to capture information about all orders of Feynman diagrammatics, the cyclic operads corresponding to the tree-level expansions.
In particular, the graph complexes introduced in \cite{Konts2} arise naturally
as the modular analog of the cobar transform, called Feynman transform in \cite{GK}.
The properties of the Feynman transform, compared to the cobar transform, are quite misterious and the calculations of homology of graph complexes constitutes
difficult combinatorial problems, the examples of which include the Vassiliev homology of the spaces of knots and the cohomology of uncompactified moduli
spaces of Riemann surfaces. 

In this note I show that the modular operads and their Feynman transforms are intimately related with
a kind of noncommutative Batalin-Vilkovisky geometry. The classical limit of this geometry is the noncommutative symplectic geometry described in \cite{kont},
\cite{Ginz} in connection with cyclic operads. 

I show,
in particular, that the algebras over the Feynman transform of an arbitrary twisted modular operad $\mathcal{P}$ are in one-to-one correspondence with solutions to quantum master equation of Batalin-Vilkovisky geometry
on affine $\mathcal{P} -$manifolds, see Theorem \ref{theorem1} from Section \ref{dgla}.

As an application I give a construction of characteristic classes with values in the homology of the quotient of Deligne-Mumford moduli
spaces. These classes are associated naturally with solutions to quantum master equation of Batalin-Vilkovisky geometry on the affine $\mathbb{S} [t] -$manifolds. The twisted $D e t -$operad $\mathbb{S} [t]$ is introduced in Section \ref{detst}. It is constructed naturally from symmetric groups and is a
generalization of the cyclic operad of associative algebras. One of the important properties of this twisted modular operad is the identification of the
stable ribbon graph complex introduced in \cite{konts3} with the Feynman
transform of $\mathbb{S} [t]$, see Theorem \ref{theorem2} in Section \ref{sectionfeynmnst}.

Here is the brief content of the sections. Sections 1-3 are introductory, sections 4-7 are devoted to the characterisation of algebras
over Feynman transform and the description of the corresponding Batalin-Vilkovisky geometry, in sections 8-11 I introduce the modular twisted $D e t -$ operad $\mathbb{S} [t]$ and prove the theorem relating the Feynman transform of $\mathbb{S} [t]$ with cell complexes of the compactified moduli spaces of Riemann surfaces. 

This work is a part of a project started
in the summer of 2000. The paper was written during my visits to Research Institut for Mathematical Sciences in Kyoto in the autumn 2003 and to Max Planck
Institut for Mathematics in Bonn in the winter 2005/2006. I'm grateful to both institutions for support and excelent working conditions. It is a pleasure
to aknowledge the stimulating discussions with Yu.I.Manin, S.Merkulov and K.Saito. I would like also to thank the referee for carefull reading of the paper.

Notations: I denote by $k$ a field of characteristic zero, if $V = \oplus _{i}V_{i}$ is a graded vector space over $k$ then $V [i]$ denotes the vector space with graded components $V [i]_{j} =V_{i +j}$, if $x \in V_{i}$ then $\bar{x} =i\ensuremath{\operatorname*{mod}}2$ denotes its degree modulo $\mathbb{Z}/2 \mathbb{Z}$, the cardinality of a finite set $I$ is denoted by $\vert I\vert $. Throughout the paper I work in the tensor symmetric category of $\mathbb{Z} -$graded vector spaces with the isomorphism $X \bigotimes Y \simeq Y \bigotimes X$
\begin{equation}x \otimes y \rightarrow ( -1)^{\bar{x} \bar{y}} y \otimes x \label{tensprod}
\end{equation}I denote via $V^{d u a l}$ the linear dual of $V$ with ($V^{d u a l})_{i} =(V_{ -i})^{d u a l}$. For a module $U$ over a finite group $G$ I denote via $U_{G}$ the $k -$vector space of coinvariants, i.e. the quotient of $U$ by submodule generated by $\{g u -u\vert u \in U ,g \in G\}$, and via $U^{G}$ the subspace of invariants: $\{ \forall g \in G :g u =u\vert u \in U\}$. For a finite set $\{V_{i}\vert i \in I\}$ of $k -$vector spaces labeled by elements of the finite set $I$ I define the tensor product \begin{equation*}\underset{i \in I}{\bigotimes }V_{i} =\left (\underset{b i j e c t i o n s :I \leftrightarrow \{1 ,\ldots  ,l\}}{\bigoplus }V_{f^{ -1} (1)} \otimes \ldots  \otimes V_{f^{ -1} (l)}\right )_{\mathbb{S}_{l}}\text{.}
\end{equation*} 

\section{Modular operads.}
In this Section I collect for reader convenience the definitions relative to the concepts of the modular operad and of the Feynman transformation
of modular operad. The material presented in this Section is borrowed from Sections 2-5 of \cite{GK}.

An $\mathbb{S}$-module $\mathcal{P}$ is a collection of chain complexes of $k -$vector spaces $\mathcal{P} ((n))$, $n \in \mathbb{N}$, equipped with an action of $\mathbb{S}_{n}$, the group of automorphisms of the set $\{1 ,\ldots  ,n\}$. Given an $\mathbb{S}$-module $\mathcal{P}$ and a finite set $I$ we extend $\mathcal{P}$ to the functor on finite sets by putting \begin{equation*}\mathcal{P} ((I)) =\left (\underset{b i j e c t i o n s :I \leftrightarrow \{1 ,\ldots  ,n\}}{\bigoplus }\mathcal{P} ((n))\right )_{\mathbb{S}_{n}}\text{.}
\end{equation*} 

A graph $G$ is a triple $(F l a g (G) ,\lambda  ,\sigma )$, where $F l a g (G)$ is a finite set, whose elements are called flags, $\lambda $ is a partition of $F l a g (G)$, and $\sigma $ is an involution acting on $F l a g (G)$. By partition here one understands a disjoint decomposition into unordered subsets, called blocks. The vertices of the graph are the
blocks of the partition. The set of vertices is denoted by $V e r t (G)$. The subset of $F l a g (G)$ corresponding to vertex $v$ is denoted by $L e g (v)$. The cardinality of $L e g (v)$ is called the valence of $v$ and is denoted $n (v)$. The edges of the graph are the pairs of flags forming a non-trvial two-cycle of the involution $\sigma $. The set of edges is denoted $E d g e (G)$. The legs of the graph are the fixed elements of the involution $\sigma $. The set of legs is denoted $L e g (G)$. The number of legs is denoted $n (G)$. I denote by $\vert G\vert $ the one-dimensional CW-complex which is the geometric realisation of the graph
$G$. 

A stable graph $G$ is a connected graph with a non-negative integer number $b (v)$ assigned to each vertex $v \in V e r t (G)$, such that $2 b (v) +n (v) -2 >0$ for any $v \in V e r t (G)$. For a stable graph $G$ I put \begin{equation*}b (G) =\Sigma _{v \in V e r t (G)} b (v) +\dim  H_{1} (\vert G\vert )\text{.}
\end{equation*} 

A stable $\mathbb{S}$-module $\mathcal{P}$ is an $\mathbb{S}$-module with extra grading by non-negative integers on each $\mathbb{S}_{n}$-module : $\mathcal{P} ((n)) = \oplus _{b \geq 0}\mathcal{P} ((n ,b))$, such that if $2 b +n -2 \leq 0$ then $\mathcal{P} ((n ,b)) =0$. I assume throughout the paper that the chain complexes $\mathcal{P} ((n ,b))$ have finite-dimensional homology. In the operad framework $\mathcal{P} ((n))$ can be thought of as the space of all possibilities to get an $n -$tensor using $\mathcal{P}$-operations. 

Given a stable $\mathbb{S}$-module $\mathcal{P}$ and a stable graph $G$ one defines \begin{equation*}\mathcal{P} ((G)) =\underset{v \in V e r t (G)}{\bigotimes }\mathcal{P} ((L e g (v) ,b (v)))\text{.}
\end{equation*} 

Let us denote by $\Gamma  ((n ,b))$ the set consisting of all pairs $(G ,\rho )$ where $G$ is a stable graph with $n (G) =n$ and $b (G) =b$ and $\rho $ is a bijection $L e g (G) \leftrightarrow \{1 ,\ldots  ,n\}$. Sometimes I shall omit from the notation for an element of $\Gamma  ((n ,b))$ the marking $\rho $ when this does not lead to a confusion. A modular operad $\mathcal{P}$ is a stable $\mathbb{S}$-module $\mathcal{P}$ together with composition maps
\begin{equation}\mu _{(G ,\rho )}^{\mathcal{P}} :\mathcal{P} ((G)) \rightarrow \mathcal{P} ((n ,b)) \label{mugamma}
\end{equation}defined for any stable graph with marked legs $(G ,\rho ) \in $ $\Gamma  ((n ,b))$ and all possible $n$ and $b$. These maps must be $\mathbb{S}_{n} -$equivariant with respect to relabeling of legs of $G$ and satisfy the natural associativity condition with respect to the compositions in the category of stable graphs (see loc.cit.
Sections 2.13-2.21). 

Given a finite set $I$ let us denote by $\Gamma  ((I ,b))$ the set consisting of stable graphs $G$ with $b (G) =b$ and with exterior legs marked by the elements of $I$. Using a bijection $I \leftrightarrow \{1 ,\ldots  ,n\}$ I extend the composition (\ref{mugamma}) to the map \begin{equation*}\mu _{G}^{\mathcal{P}} :\mathcal{P} ((G)) \rightarrow \mathcal{P} ((I ,b))
\end{equation*} defined for $G \in \Gamma  ((I ,b))$. Because of $\mathbb{S}_{n} -$equivariance this does not depend on the choice of the bijection.The associativity condition satisfied by
the compositions $\mu _{G}^{\mathcal{P}}$ can be described as follows. For any subset of edges $J \subseteq E d g e (G)$ one has the stable graph $G/J$ and naturally defined morphism of stable graphs $f :G \rightarrow G/J$. For such morphism one defines the natural map \begin{equation*}\mu _{G \rightarrow G/J}^{\mathcal{P}} :\mathcal{P} ((G)) \rightarrow \mathcal{P} ((G/J)) ,\text{\thinspace \thinspace \thinspace \thinspace }\mu _{G \rightarrow G/J}^{\mathcal{P}} =\underset{v \in V e r t (G/J)}{\bigotimes }\mu _{f^{ -1} (v)}^{\mathcal{P}}
\end{equation*} and the associativity condition tells that \begin{equation*}\mu _{G}^{\mathcal{P}} =\mu _{G/J}^{\mathcal{P}} \circ \mu _{G \rightarrow G/J}^{\mathcal{P}}
\end{equation*} 

It follows from the asssociativity condition, that it is sufficient to define the compositions \ref{mugamma}
just for the stable graphs with one edge. There are two types of such graphs. The first one, which I denote $G_{(I_{1} ,I_{2} ,b_{1} ,b_{2})}$, has two vertices so that the set of legs is decomposed into two subsets $I_{1} \sqcup I_{2} =\{1 ,\ldots  ,n\}$. The composition along $G_{(I_{1} ,I_{2} ,b_{1} ,b_{2})}$ has the form
\begin{equation}\mu _{G_{(I_{1} ,I_{2} ,b_{1} ,b_{2})}}^{\mathcal{P}} :\mathcal{P} ((I_{1} \sqcup \{f\} ,b_{1})) \otimes \mathcal{P} ((I_{2} \sqcup \{f^{ \prime }\} ,b_{2})) \rightarrow \mathcal{P} ((n ,b_{1} +b_{2})) \label{phiij}
\end{equation}where $f ,f^{ \prime }$ are the two flags corresponding to the edge joining the two vertices. Actually, since the symmetric group $\mathbb{S}_{n}$, acting on $\mathcal{P} ((n ,b_{1} +b_{2}))\text{,}$ acts transitevely on the set of pairs $I_{1}$, $I_{2}$ in (\ref{phiij}) with fixed cardinality, it is sufficient to consider the composition (\ref{phiij})
just for the subsets $I_{1} =\{1 ,\ldots  ,m -1\}$, $I_{2} =\{m ,\ldots  ,n\}$. The second type of the stable graphs with one edge, which I denote by $G_{n ,b}$, has one vertice and its single edge is a loop. The composition along $G_{n ,b}$ is
\begin{equation}\mu _{G_{n ,b}}^{\mathcal{P}} :\mathcal{P} ((\{1 ,\ldots  n\} \sqcup \{f ,f^{ \prime }\} ,b -1)) \rightarrow \mathcal{P} ((n ,b)) \label{dzetaij}
\end{equation}where $f$ and $f^{ \prime }$ are the flags corresponding to the unique edge. 

An example of modular operad is the endomorphism operad of a
chain complex of $k -$vector spaces $V = \oplus _{i}V_{i} [ -i]$ equipped with symmetric pairing $B$ of degree $0$, \begin{equation*}B (u ,v) =( -1)^{\bar{u}\; \bar{v}} B (v ,u) ,\text{\thinspace \thinspace \thinspace \thinspace }B :V^{ \otimes 2} \rightarrow k
\end{equation*} so that $B (u ,v) =0$ unless $\deg u +\deg  v =0$. The $\mathbb{S} -$module underlying the endomorphism modular operad of $V$ is defined by
\begin{equation}\mathcal{E} [V] ((n ,b)) =V^{ \otimes n} \label{evnb}
\end{equation}with the standard $\mathbb{S}_{n} -$action. Then \begin{equation*}\mathcal{E} [V] ((G)) =V^{ \otimes F l a g (G)}
\end{equation*} The composition (\ref{mugamma}) is the contraction with $B^{ \otimes E d g e (G)}$. Because $\deg B =0$, this is compatible with the definition of the usual endomorphisms whose components are defined by $H o m_{k} (V^{ \otimes n -1} ,V)$: the isomorphism induced by $B$: $V$ $ \simeq V^{d u a l}$ gives the isomorphisms of the underlying operad\begin{equation*}\mathcal{E} [V] ((n ,b)) \simeq H o m_{k} (V^{ \otimes n -1} ,V)\text{.}
\end{equation*} 

Another series of examples is given by cyclic operads with $\mathcal{P} ((m)) =0$ for $m =1 ,2$, which can be considered as modular operads by putting $\mathcal{P} ((m ,b)) =0$ for $b \geq 1$. 

\bigskip The image of a modular operad under Feynman transform is some modification
of modular operad with extra signs involved. To take into account these signs one needs to introduce the twisting of modular operads. The twisting is also
unavoidable when one wishes to associate an endomorphism modular operad with chain complex with symmetric or antisymmetric inner products of arbitrary degree.

\subsection{\bigskip Determinants}
To simplify the signs bookkeeping it is convenient to introduce for a finite dimensional $k -$vector space $V$ the determinant \begin{equation*}D e t (V) =\Lambda ^{\dim V} (V) [\dim V]\text{.}
\end{equation*} This is the top-dimensional exterior power of the $k -$vector space $V$ concentrated in degree $( -\dim  V)$. I shall mostly need the determinant of the vector space $k^{S}$ associated with a finite set $S$. I denote it $D e t (S)$: \begin{equation*}D e t (S) =D e t (k^{S})\text{.}
\end{equation*} Because of (\ref{tensprod}) one has the natural isomorphism for the
disjoint union of sets $ \sqcup _{i \in I}S_{i}$
\begin{equation}D e t (\coprod _{i \in I}S_{i}) \simeq \underset{i \in I}{\bigotimes }D e t (S_{i}) . \label{detsi}
\end{equation}Another obvious property is $D e t^{ \otimes 2} (S) \simeq k [2 \vert S\vert ]$. I shall also put for a graded finite dimensional $k -$vector space $V_{ \ast }$ \begin{equation*}D e t (V_{ \ast }) =\underset{j \in \mathbb{Z}}{\bigotimes }D e t (V_{j})^{(( -1)^{j\ensuremath{\operatorname*{mod}}2})}\text{.}
\end{equation*} 

\subsection{\bigskip Cocycles.}
I will only consider cocycles with values in the Picard tensor symmetric category of invertible graded $k -$vector spaces. Such a cocycle $\mathcal{D}$ is a functor which assigns to a stable graph $G$ a graded one-dimensional vector space $\mathcal{D} (G)$ and to any morphism of stable graphs $f :G \rightarrow G/J$ the linear isomorphism \begin{equation*}\nu _{f} :\mathcal{D} (G/J) \otimes \underset{v \in V e r t (G/J)}{\bigotimes } \mathcal{D} (f^{ -1} (v)) \rightarrow \mathcal{D} (G)
\end{equation*} satisfying the natural associativity condition with respect to the composition of two morphisms, see loc.cit. Section
4.1. For the graph with only one vertice and no edges $G = \ast _{n ,b}$I must have $\mathcal{D} ( \ast _{n ,b}) =k$. Examples of such cocycles are
\begin{equation}\mathcal{K} (G) =D e t (E d g e (G)) \label{kG}
\end{equation}\begin{equation*}\mathcal{L} (G) =D e t (F l a g (G)) D e t^{ -1} (L e g (G))\text{.}
\end{equation*} 

\subsection{\bigskip Twisted modular operads.}
\bigskip A twisted modular $\mathcal{D} -$operad $\mathcal{P}$ is a stable $\mathbb{S}$-module $\mathcal{P}$, $\mathcal{P} ((n)) = \oplus _{b}$ $\mathcal{P} ((n ,b))$, together with composition maps
\begin{equation}\mu _{(G ,\rho )}^{\mathcal{P}} :\mathcal{D} (G) \otimes \mathcal{P} ((G)) \rightarrow \mathcal{P} ((n ,b)) \label{mugammaK}
\end{equation}defined for $(G ,\rho ) \in \Gamma  ((n ,b))$ which should satisfy the $\mathbb{S}_{n} -$equivariance and the associativity conditions parallel to that of modular operad. 

\subsection{\bigskip Coboundaries.}
Let $s$ be a stable $\mathbb{S}$-module such that $\dim _{k} s ((g ,n)) =1$ for all $g ,n$. Then $s$ defines a cocycle
\begin{equation}\mathcal{D}_{s} (G) =s ((n ,b)) \otimes \underset{v \in V e r t (G)}{\bigotimes } s^{ -1} ((n (v) ,b (v))) . \label{DsG}
\end{equation}This is called the coboundary of $s$. Tensoring underlying $\mathbb{S}$-modules by $s$ defines equivalence of the category of modular $\mathcal{D} -$operad with the category of modular $\mathcal{D} \otimes \mathcal{D}_{s} -$operad. Examples of such coboundaries are \begin{equation*}\Sigma  ((n ,b)) =k [ -1]
\end{equation*}\begin{equation*}\alpha  ((n ,b)) =k [n]
\end{equation*}\begin{equation*}\beta  ((n ,b)) =k [b -1]
\end{equation*}\begin{equation*}\tilde{\mathfrak{s}} =s g n_{n} [n]
\end{equation*} in the first three examples the $\mathbb{S}_{n} -$action is trivial, and in the last example it is the alternating representation. 

\subsection{Free modular operads.}
The forgetful functor \begin{equation*}\;\text{\emph{modular operads}}\; \rightarrow \;\text{\emph{stable}}\;\;\mathbb{S} -\;\text{\emph{modules}}\;
\end{equation*}
has the left adjoint functor which associates to a stable $\mathbb{S} -$module $\mathcal{A}$ the free modular operad $\mathbb{M} \mathcal{A}$ generated by $\mathcal{A}$:\begin{equation*}\mathbb{M} \mathcal{A} ((n ,b)) =\underset{G \in [\Gamma ((n ,b))]}{\bigoplus }\mathcal{A} ((G))_{A u t (G)}
\end{equation*} where $[\Gamma ((n ,b))]$ denotes the set of isomorphisms classes of pairs $(G ,\rho )$ where $G$ is a stable graph with $n (G) =n$, $b (G) =b$ and $\rho $ is a bijection $L e g (G) \leftrightarrow \{1 ,\ldots  n\}$. 

Similarly one defines the free modular twisted $\mathcal{D} -$operad $\mathbb{M}_{\mathcal{D}} \mathcal{A}$ generated by stable $\mathbb{S} -$module $\mathcal{A}$: \begin{equation*}\mathbb{M}_{\mathcal{D}} \mathcal{A} ((n ,b)) =\underset{G \in [\Gamma ((n ,b))]}{\bigoplus }(\mathcal{D} (G) \otimes \mathcal{A} ((G)))_{A u t (G)}\text{.}
\end{equation*} On the subspace of generators the composition map $\mu _{G}$ is simply the projection $\mathcal{D} (G) \otimes \mathcal{A} ((G)) \rightarrow (\mathcal{D} (G) \otimes \mathcal{A} ((G)))_{A u t (G)}\text{.}$ 

\subsection{Feynman transform.}
Let $\mathcal{P}$ be a modular $\mathcal{D} -$operad. I assume for simplicity below that all spaces $\mathcal{P} ((n ,b))$ are finite-dimensional in each degree. One can avoid this assumption in the standard way by introducing the modular cooperads, see
\cite{GJ}, I leave details to an interested reader. In our examples below
the spaces $\mathcal{P} ((n ,b))$ are finite-dimensional in each degree. Let us put $\mathcal{D}^{ \vee } =\mathcal{K} \mathcal{D}^{ -1}$, where $\mathcal{K}$ is the cocycle (\ref{kG}). The Feynman transform of a modular $\mathcal{D} -$operad $\mathcal{P}$ is a modular $\mathcal{D}^{ \vee } -$operad $\mathcal{F}_{\mathcal{D}} \mathcal{P}$, defined in the following way. As a stable $\mathbb{S} -$module, forgetting the differential, $\mathcal{F}_{\mathcal{D}} \mathcal{P}$ is the free modular $\mathcal{D}^{ \vee } -$operad generated by stable $\mathbb{S} -$module $\{\mathcal{P} ((n ,b))^{d u a l}\}$. The differential on $\mathcal{F}_{\mathcal{D}} \mathcal{P}$ is the sum $d_{\mathcal{F}} = \partial _{\mathcal{P}^{d u a l}} + \partial _{\mu }$ of the differential $ \partial _{\mathcal{P}^{d u a l}}$ induced on $\mathbb{M}_{\mathcal{D}^{ \vee }} \mathcal{P}^{d u a l}$ by the differential on $\mathcal{P}\,$and of the differential $ \partial _{\mu }$, whose value on the term $(\mathcal{D}^{ \vee } (G) \otimes \mathcal{P}^{d u a l} ((G)))_{A u t (G)}$ is a sum over all equivalence classes of stable graphs $\tilde{G}$ such that $\tilde{G}/\{e\} \simeq G$ of the map dual to the composition $\mu _{\tilde{G} \rightarrow G}^{\mathcal{P}}$ multiplied by the element $e [1] \in D e t (\{e\})$:\begin{equation*} \partial _{\mu }\vert _{(\mathcal{D}^{ \vee } (G) \otimes \mathcal{P}^{d u a l} ((G)))_{A u t (G)}} =\sum _{\tilde{G}/\{e\} \simeq G}e [1] \otimes (\mu _{\tilde{G} \rightarrow G}^{\mathcal{P}})^{d u a l}
\end{equation*} see Section~5 of loc.cit. 

The Feynman transform is a generalization of graph
complexes from \cite{kont}. The Lie, commutative and associative graph complexes
correspond to the $n (G) =0$ part of the Feynman transforms of the corresponding cyclic operads, considered as modular operads with $\mathcal{P} ((n ,b)) =0$ for $b \geq 1$. 

\section{Endomorphisms operad with inner products of degree $l \in \mathbb{Z}$.}
Here I discuss the natural twisted modular operads of endomorphisms associated with symmetric or antisymmetric inner products of degree $l \in \mathbb{Z}$. This is a relatively straightforward extension of the degree zero symmetric and degree $ -1$ antisymmetric cases described in \cite{GK},
subsections 2.25 and 4.12. 

\subsection{\bigskip Symmetric inner product of degree $l \in \mathbb{Z}$.}
Let $V$ be a chain complex with symmetric inner product $B$ of arbitrary degree, $\deg B =l$, $l \in \mathbb{Z}$:\begin{equation*}B (u ,v) =( -1)^{\bar{u}\; \bar{v}} B (v ,u) ,\text{\thinspace \thinspace \thinspace \thinspace }B :V^{ \otimes 2} \rightarrow k [ -l] ,\text{\thinspace \thinspace }l \in \mathbb{Z}
\end{equation*} so that $B (u ,v) =0$ unless $\deg u +\deg  v =l$. If I put for underlying $\mathbb{S}$-modules \begin{equation*}\mathcal{E} [V] ((n ,b)) =V^{ \otimes n}
\end{equation*} then the contraction with $B^{E d g e (G)}$defines naturally the composition map \begin{equation*}\mu _{G}^{\mathcal{E} [V]} :\mathcal{K}^{ \otimes l} (G) \otimes \mathcal{E} [V] ((G)) \rightarrow \mathcal{E} [V] ((n ,b))
\end{equation*} of the modular $\mathcal{K}^{ \otimes l} -$operad where \begin{equation*}\mathcal{K}^{ \otimes l} (G) =D e t^{ \otimes l} (E d g e (G))
\end{equation*} Indeed, for even $l$, $l =2 l^{ \prime }$, the tensoring by the cocycle acts simply as the degree shift \begin{equation*}\mathcal{K}^{ \otimes 2 l^{ \prime }} (G) =k[2 l^{ \prime }\vert E d g e(G)\vert ]
\end{equation*} and the contraction with $B^{E d g e (G)}$ acting on $V^{F l a g (G)}$ decreases the total degree exactly by $2 l^{ \prime } \vert E d g e (G)\vert $. For odd $l$, $l =2 l^{ \prime } +1$, the cocycle is the degree shift tensored by the top exterior power of $k^{E d g e (G)}$ \begin{equation*}\mathcal{K}^{ \otimes 2 l^{ \prime } +1} (G) =\Lambda ^{\vert E d g e (G)\vert } (k^{E d g e (G)})[(2 l^{ \prime } +1)\vert E d g e (G)\vert ]\text{.}
\end{equation*} Notice that the permutation of any two edges inverses the sign of the value of $B^{ \otimes E d g e (G)}$ on $V^{ \otimes F l a g (G)}$ since $B$ is of odd degree. This explains the necessity for the term $\Lambda ^{\vert E d g e (G)\vert } (k^{\vert E d g e (G)\vert })$ in the case of odd degree. 

\subsection{Antisymmetric inner product of degree $l \in \mathbb{Z}$}
Let $V$ be a chain complex with \emph{antisymmetric} inner product \begin{equation*}B (u ,v) = -( -1)^{\bar{u}\; \bar{v}} B (v ,u) ,\;\deg B =l ,l \in \mathbb{Z}\text{.}
\end{equation*} The cocycle corresponding to such inner product is \begin{equation*}\mathcal{K}^{ \otimes l -2} \mathcal{L} (G) =D e t^{ \otimes l -2} (E d g e (G)) D e t (F l a g (G)) D e t^{ -1} (L e g (G))\text{.}
\end{equation*} Using (\ref{detsi}) I see that \begin{equation*}\mathcal{K}^{ \otimes l -2} \mathcal{L} (G) =\mathcal{K}^{ \otimes l} \otimes (\underset{e \in E d g e (G)}{\bigotimes }\Lambda ^{2} (k^{\{s_{e} ,t_{e}\}}))
\end{equation*} where $\{s_{e} ,t_{e}\}$ is the set of two flags corresponding to the edge $e$. I put now $\mathcal{E} [V] ((n ,b)) =V^{ \otimes n}$ and define the composition map (\ref{mugammaK}) as in the previous cases: identify $\mathcal{E} [V] ((G))$ with $V^{F l a g (G)}$ and contract with $B^{E d g e (G)}$. The composition map is well defined since permutation of two flags $s_{e} ,t_{e}$ reverses the sign of $\Lambda ^{2} (k^{\{s_{e} ,t_{e}\}})$.

\begin{definition}
Twisted modular $\mathcal{P}$-algebra structure on chain complex $V$ with symmetric (respectively antisymmetric) inner product of degree $l$ is a morphism of twisted modular $\mathcal{K}^{ \otimes l} -$operads (respectively $\mathcal{K}^{ \otimes l -2} \mathcal{L} -$operads): $\mathcal{P} \rightarrow \mathcal{E} [V]$. 
\end{definition}

\subsection{Suspension.}
The suspension coboundary $\mathcal{D}_{\mathfrak{s}}$, defined in \cite{GK}, subsection 4.4, is associated
with the $\mathbb{S} -$module \begin{equation*}\mathfrak{s} ((n ,b)) =s g n_{n} [2(b -1) +n]
\end{equation*} where $s g n_{n}$ is the standard alternating representation of $\mathbb{S}_{n}$. If I identify in the tensor product (\ref{DsG}) $s g n_{n (v)}[n(v)]$ with $D e t (L e g (v))$ then I get \begin{equation*}\mathcal{D}_{\mathfrak{s}} (G) =D e t (L e g (G)) [2(b(G) -1)]\underset{v \in V e r t (G)}{\bigotimes }D e t^{ -1} (L e g (v))[2(1 -b(v))]\text{.}
\end{equation*} This is equal to \begin{equation*}\mathcal{D}_{\mathfrak{s}} (G) =\mathcal{L}^{ -1}[2\vert E d g e(G)\vert ] =\mathcal{L}^{ -1} \mathcal{K}^{ \otimes 2} (G)
\end{equation*} because of the formula \begin{equation*}\Sigma _{v \in V e r t (G)} (b (v) -1) =b (G) -1 -\vert E d g e (G)\vert 
\end{equation*} I see that the multiplication by $\mathfrak{s}$ transforms modular $\mathcal{K}^{ \otimes l -2} \mathcal{L}$-operads to modular $\mathcal{K}^{ \otimes l}$-operads and vice versa. In particular, for algebras over such operads the degrees of the corresponding inner products must be the same.
The coboundary associated with \begin{equation*}\tilde{\mathfrak{s}} =s g n_{n} [n]
\end{equation*} so that\begin{equation*}\mathcal{D}_{\tilde{\mathfrak{s}}} =\mathcal{L}^{ -1}
\end{equation*} is sometimes more useful in the situation of inner products of arbitrary degree. The multiplication by $\tilde{\mathfrak{s}}$ of the underlying $\mathbb{S} -$module transforms modular $\mathcal{K}^{ \otimes l -2} \mathcal{L}$-operad to modular $\mathcal{K}^{ \otimes l -2}$-operad and, since $\mathcal{L}^{2} \simeq \mathcal{K}^{ \otimes 4}$, it transforms modular $\mathcal{K}^{ \otimes l +2}$-operad to modular $\mathcal{K}^{ \otimes l -2} \mathcal{L}$-operad. If $V$ is a chain complex with symmetric (respectively antisymmetric) inner product $B$ of degree $\deg B =l$, then the suspension of $V$ is a chain complex $V [1]$ with the antisymmetric (respectively symmetric) inner product $\tilde{B}$ of degree $l -2$ defined by \begin{equation*}\tilde{B}(x[1] ,y [1]) =( -1)^{\bar{x}} B (x ,y)\text{.}
\end{equation*} The multiplication by $\tilde{\mathfrak{s}}$ of the modular $\mathcal{K}^{ \otimes l}$-operad $\mathcal{E} [V]$ gives the modular $\mathcal{K}^{ \otimes l -4} \mathcal{L}$-operad $\mathcal{E} [V [1]]$. It follows that if $\mathcal{P}$ is a $\mathcal{K}^{ \otimes l} -$operad and $V$ is a chain complex with symmetric inner product $B$ of degree $l$, then the modular $\mathcal{P} -$algebra structure on $V$ corresponds under the suspension to the modular $\tilde{\mathfrak{s}} \mathcal{P}$-algebra structure on the chain complex $V [1]$ equipped with the antisymmetric inner product $\tilde{B}$ of degree $l -2$. 

\subsection{Reducing the twistings.}
The twistings corresponding to the inner products can be reduced using coboundaries to just two cocycles: the trivial cocycle, if $l$ is even, and the cocycle $D e t$ whose value on a graph $G$ is \begin{equation*}D e t (G) =D e t (H_{1} (G))\text{,}
\end{equation*} if $l$ is odd. The last cocycle is isomorphic to \begin{equation*}D e t (G) \simeq \mathcal{K}^{ -1} \mathcal{D}_{\tilde{\mathfrak{s}}}^{ -1} \mathcal{D}_{\Sigma }^{ -1}
\end{equation*} see loc.cit. Proposition 4.14. Notice that both cocycles are trivial on trees. Possible choices of the coboundaries reducing
the twistings to the trivial or $D e t$ is given by the following identities :
\begin{gather*}(\mathcal{K}^{ \otimes 2 l}) \mathcal{D}_{\beta }^{ \otimes  -2 l} \simeq k ,\;(\mathcal{K}^{ \otimes 2 l} \mathcal{L}) \mathcal{D}_{\tilde{\mathfrak{s}}} \mathcal{D}_{\beta }^{ \otimes  -2 l} \simeq k\text{,} \\
(\mathcal{K}^{ \otimes 2 l -1}) \mathcal{D}_{\tilde{\mathfrak{s}}}^{ -1} \mathcal{D}_{\beta }^{ \otimes  -2 l} \mathcal{D}_{\Sigma }^{ -1} \simeq D e t (G)\text{\ \ }(\mathcal{K}^{ \otimes 2 l -1} \mathcal{L}) \mathcal{D}_{\beta }^{ \otimes  -2 l} \mathcal{D}_{\Sigma }^{ -1} \simeq D e t (G)\text{.}\end{gather*}

Let $\mathcal{P}$ be a cyclic operad. Putting $\mathcal{P} ((n ,b)) =0$ for $b >0$, the cyclic operad can be considered both as a modular operad and as a twisted modular $D e t -$operad. Tensoring the cyclic operad by the coboundaries as above one can obtain the twisted $\mathcal{K}^{ \otimes l} -$operad or $\mathcal{K}^{ \otimes l -2} \mathcal{L} -$operad. In such a way one can define for any cyclic operad the notion of $\mathcal{P} -$algebra on complexes with symmetric or antisymmetric inner products of arbitrary degree. 

\section{Algebras over Feynman transform.\label{sectFeynm}}
\bigskip In this Section I write down in several equivalent forms the equation defining the structure of algebra over
Feynman transform for an arbitrary twisted modular $\mathcal{D} -$operad. As I show in Section \ref{dgla} this is in fact a Maurer-Cartan
equation in the differential graded Lie algebra of geometric origin. 

Let us consider first the case of $\mathcal{F}_{\mathcal{D}} \mathcal{P} -$algebra structure on the chain complex $V$ with symmetric inner product $B$ of degree $l$, $B :V^{ \otimes 2} \rightarrow k [ -l]$, this implies that $\mathcal{D} \simeq \mathcal{K}^{ \otimes 1 -l}$. The $\mathcal{F}_{\mathcal{D}} \mathcal{P} -$a$\lg $ebra structure on $V$ is a morphism of twisted modular $\mathcal{K}^{ \otimes l} -$operads $\widehat{m} :\mathcal{F}_{\mathcal{D}} \mathcal{P} \rightarrow \mathcal{E} [V]$. Since, forgetting the differential, $\mathcal{F}_{\mathcal{D}} \mathcal{P}$ is the free twisted modular operad generated by stable $\mathbb{S} -$module $\mathcal{P} ((n ,b))^{d u a l}$, the $\mathcal{F}_{\mathcal{D}} \mathcal{P}$-algebra structure on $V$ is determined by a set of $\mathbb{S}_{n} -$equivariant linear maps \begin{equation*}\widehat{m}_{n ,b} :\mathcal{P} ((n ,b))^{d u a l} \rightarrow V \text{}^{ \otimes n}
\end{equation*} or, equivalently, the set of degree zero elements
\begin{equation}m_{n ,b} \in (V \text{}^{ \otimes n} \otimes \mathcal{P} ((n ,b)))^{\mathbb{S}_{n}} . \label{mnb}
\end{equation}As above, for any finite set $I$ one can extend this to the collection of elements $\{\widehat{m}_{I ,b}\}$: $\widehat{m}_{I ,b} \in H o m (\mathcal{P} ((I ,b))^{d u a l} ,V \text{}^{ \otimes I})$, and $\{m_{I ,b}\}$: $m_{I ,b} \in (V \text{}^{ \otimes I} \otimes \mathcal{P} ((I ,b)))^{A u t (I)}$, using an arbitrary bijection $I \leftrightarrow \{1 ,\ldots  ,\vert I\vert \}$. An element from the subspace
\begin{equation}(\mathcal{K}^{ \otimes l} (G) \otimes \mathcal{P}^{d u a l} ((G)))_{A u t (G)} \subset \mathcal{F}_{\mathcal{D}} \mathcal{P} ((n ,b)) \label{subFdP}
\end{equation}corresponding to a stable graph $G \in [\Gamma ((n ,b))]$, is represented as the result of composition $\mu _{G}^{\mathcal{F}_{\mathcal{D}} \mathcal{P}}$acting on an element from $\underset{v \in V e r t (G)}{\bigotimes }\mathcal{P} ((n (v) ,b (v)))^{d u a l}$. It follows that on the subspace (\ref{subFdP}) the map \begin{equation*}\widehat{m} :\mathcal{F}_{\mathcal{D}} \mathcal{P} ((n ,b)) \rightarrow V \text{}^{ \otimes n}
\end{equation*} is given by
\begin{equation}\mu _{G}^{\mathcal{E} [V]} \circ \left (\underset{v \in V e r t (G)}{\bigotimes }\widehat{m}_{L e g (v) ,b (v)}\right ) . \label{mhatG}
\end{equation}The map $\widehat{m} :\mathcal{F}_{\mathcal{D}} \mathcal{P} \rightarrow \mathcal{E} [V]$ corresponding to the set $\{m_{n ,b}\}$ is a morphism of twisted modular $\mathcal{K}^{ \otimes l}$-operads if and only if
\begin{equation}d_{\mathcal{E} [V]} \circ \widehat{m} =\widehat{m} \circ d_{\mathcal{F}_{\mathcal{D}} \mathcal{P}} . \label{dmmd}
\end{equation}Since the differentials $d_{\mathcal{F}_{\mathcal{D}} \mathcal{P}}$, $d_{\mathcal{E} [V]}$ are compatible with the composition maps $\mu _{G}^{\mathcal{F}_{\mathcal{D}} \mathcal{P}}$, $\mu _{G}^{\mathcal{E} [V]}$, it is sufficient to check the condition (\ref{dmmd}) on the generators of $\mathcal{F}_{\mathcal{D}} \mathcal{P}$. On the subspace $\mathcal{P} ((n ,b))^{d u a l}$ the differential $d_{\mathcal{F}_{\mathcal{D}} \mathcal{P}}$ is a sum of $d_{\mathcal{P}^{d u a l}}$ plus sum of the adjoints to the structure maps $\mu _{G}^{\mathcal{P}}$ (\ref{mugammaK}) corresponding to the stable graphs with single edge $G_{(I_{1} ,I_{2} ,b_{1} ,b_{2})}$ and $G_{n ,b}$, multiplied by $e [1]$, the canonical element of degree (-1) from $D e t (\{e\})$ where $e$ is the unique edge of the graph $G$:\begin{equation*}d_{\mathcal{F}_{\mathcal{D}} \mathcal{P}} =d_{\mathcal{P}^{d u a l}} +e [1] \otimes (\mu _{G_{n ,b}}^{\mathcal{P}})^{d u a l} +\sum _{\{1 ,\ldots  ,n\} =I_{1} \sqcup I_{2} ,\;b_{1} +b_{2} =b}e [1] \otimes (\mu _{G_{(I_{1} ,I_{2} ,b_{1} ,b_{2})}}^{\mathcal{P}})^{d u a l}\text{.}
\end{equation*} We see that the condition (\ref{dmmd}) is equivalent to
\begin{multline}d_{V^{ \otimes n}} \widehat{m}_{n ,b} =\widehat{m}_{n ,b} d_{\mathcal{P}^{d u a l}} +\mu _{G_{n ,b}}^{\mathcal{E} [V]} (e[1] \otimes \widehat{m}_{\{1 ,\ldots  ,n\} \sqcup \{f ,f^{ \prime }\} ,b -1} \mu _{G_{n ,b}}^{\mathcal{P}} \text{}^{d u a l} ) + \\
 +\frac{1}{2} \sum _{\substack{\{1 ,\ldots  ,n\} =I_{1} \sqcup I_{2} \\ b_{1} +b_{2} =b}}\mu _{G_{(I_{1} ,I_{2} ,b_{1} ,b_{2})}}^{\mathcal{E} [V]} (e[1] \otimes (\widehat{m}_{I_{1} \sqcup \{f\} ,b_{1}} \otimes \widehat{m}_{I_{2} \sqcup \{f^{ \prime }\} ,b_{2}}) \mu _{G_{(I_{1} ,I_{2} ,b_{1} ,b_{2})}}^{\mathcal{P}} \text{}^{d u a l}) . \label{modAinfhat}\end{multline}Recall that $\mathcal{E} [V]$ is a $K^{ \otimes l} -$operad and \begin{equation*}\mathcal{K}^{ \otimes l} (G_{n ,b}) =\mathcal{K}^{ \otimes l} (G_{(I_{1} ,I_{2} ,b_{1} ,b_{2})}) =(k[1])^{ \otimes l}\text{.}
\end{equation*} Then $\mu _{G_{n ,b}}^{\mathcal{E} [V]}$ is the contraction $(V^{ \otimes \{f ,f^{ \prime }\}} \otimes V^{ \otimes n}) [l] \rightarrow V^{ \otimes n}$ with the bilinear form $B$ applied to the factors corresponding to $f ,f^{ \prime }$ and $\mu _{G_{(I_{1} ,I_{2} ,b_{1} ,b_{2})}}^{\mathcal{E} [V]}$ is the similar contraction $(V^{ \otimes I_{1} \sqcup \{f\}} \otimes V^{ \otimes I_{2} \sqcup \{f^{ \prime }\}}) [l] \rightarrow V^{ \otimes n}$. I denote these contractions by $B_{f ,f^{ \prime }}$. If I introduce the degree $(l -1)$ maps, which are the evaluation on $(e[1])^{ \otimes 1 -l}$ of the $\mathcal{P} -$compositions $\mu _{G}^{\mathcal{P}}$
\begin{multline}\phi _{f ,f^{ \prime }}^{\mathcal{P}} :\mathcal{P} ((I_{1} \sqcup \{f\} ,b_{1})) \otimes \mathcal{P} ((I_{2} \sqcup \{f^{ \prime }\} ,b_{2})) \rightarrow \mathcal{P} ((n ,b_{1} +b_{2})) [l -1] \label{phipffK} \\
\phi _{f ,f^{ \prime }}^{\mathcal{P}} =\mu _{G_{(I_{1} ,I_{2} ,b_{1} ,b_{2})}}^{\mathcal{P}} (e[1]^{ \otimes 1 -l})\end{multline}
\begin{multline}\xi _{f ,f^{ \prime }}^{\mathcal{P}} :\mathcal{P} ((\{1 ,\ldots  n\} \sqcup \{f ,f^{ \prime }\} ,b -1)) \rightarrow \mathcal{P} ((n ,b)) [l -1] \label{dsetaffK} \\
\xi _{f ,f^{ \prime }}^{\mathcal{P}} =\mu _{G_{n ,b}}^{\mathcal{P}} (e[1]^{ \otimes 1 -l})\end{multline}then in terms of $m_{I ,b}$ the equation (\ref{modAinfhat}) is written as
\begin{multline}(d_{\mathcal{P}} +d_{V}) m_{n ,b} -B_{f ,f^{ \prime }} \otimes \xi _{f ,f^{ \prime }}^{\mathcal{P}} m_{\{1 ,\ldots  ,n\} \sqcup \{f ,f^{ \prime }\} ,b -1} \\
 -\frac{1}{2} \sum _{\substack{\{1 ,\ldots  ,n\} =I_{1} \sqcup I_{2} \\ b_{1} +b_{2} =b}}B_{f ,f^{ \prime }} \otimes \phi _{f ,f^{ \prime }}^{\mathcal{P}} (m_{I_{1} \sqcup \{f\} ,b_{1}} \otimes m_{I_{2} \sqcup \{f^{ \prime }\} ,b_{2}}) =0. \label{modulAinftyEq}\end{multline}Let us put $m_{n} =\sum _{b}z^{b} m_{n ,b}$ then
\begin{multline}(d_{\mathcal{P}} +d_{V}) m_{n} -z B_{f ,f^{ \prime }} \otimes \xi _{f ,f^{ \prime }}^{\mathcal{P}} m_{\{1 ,\ldots  ,n\} \sqcup \{f ,f^{ \prime }\}} \\
 -\frac{1}{2} \sum _{\{1 ,\ldots  ,n\} =I_{1} \sqcup I_{2}}B_{f ,f^{ \prime }} \otimes \phi _{f ,f^{ \prime }}^{\mathcal{P}} (m_{I_{1} \sqcup \{f\}} \otimes m_{I_{2} \sqcup \{f^{ \prime }\}}) =0. \label{modulAinflambda}\end{multline}

For the sake of simplicity I can rewrite this equation directly in terms of $\{m_{n}\}$ using the canonical projection $\frac{1}{n !} (\Sigma _{\sigma  \in \mathbb{S}_{n}} \sigma )$ to the $\mathbb{S}_{n} -$invariant subspace. Notice that the terms of the last summand in (\ref{modulAinftyEq})
are invariant with respect to the action of the subgroup $\mathbb{S}_{c a r d (I_{1})} \times \mathbb{S}_{c a r d (I_{2})}$. The $\mathbb{S}_{n}$-equivariance of composition maps (\ref{mugamma}) implies that the result of the action of arbitrary
element $\sigma $ of $\mathbb{S}_{n}$ on such a term is $(B_{f ,f^{ \prime }} \phi _{\tilde{I}_{1} \tilde{I}_{2}}) (m_{\tilde{I}_{1} \sqcup \{f\} ,b_{1}} \otimes m_{\tilde{I}_{2} \sqcup \{f^{ \prime }\} ,b_{2}})$ with $\tilde{I}_{1} =\sigma  (I_{1})$, $\tilde{I}_{2} =\sigma  (I_{2})$. Let us single out the term $B_{f ,f^{ \prime }} \phi _{f ,f^{ \prime }}^{\mathcal{P}} (m_{I_{1} \sqcup \{f\}} \otimes m_{I_{2} \sqcup \{f^{ \prime }\}})$ with $I_{1} =\{1 ,\ldots  ,n_{1}\}$ and denote via $o^{\mathcal{P}}$ the composition $\phi _{f ,f^{ \prime }}^{\mathcal{P}}$after the identification of $I_{1} \sqcup \{f\}$ with $\{1 ,\ldots  ,n_{1} +1\}$ such that $i \leftrightarrow i$ for $1 \leq i \leq n_{1}$ and $f \leftrightarrow n_{1} +1$, and the identification of $I_{2} \sqcup \{f^{ \prime }\}$ with $\{1 ,\ldots  ,n_{2} +1\}$ such that $f^{ \prime } \leftrightarrow 1$, $i \leftrightarrow i -n_{1} +1$ for $n_{1} +1 \leq i \leq n$ : \begin{equation*}o^{\mathcal{P}} :\mathcal{P} ((n_{1} +1 ,b)) \otimes \mathcal{P} ((n -n_{1} +1 ,b^{ \prime })) \rightarrow \mathcal{P} ((n ,b +b^{ \prime })) [l -1]\text{.}
\end{equation*} Let us identify also in the first summand $\{1 ,\ldots  ,n\} \sqcup \{f ,f^{ \prime }\}$ with $\{1 ,\ldots  ,n +2\}$ in such a way that $i \leftrightarrow i$ for $i \in \{1 ,\ldots  ,n\}$ and $f \leftrightarrow n +1$, $f^{ \prime } \leftrightarrow n +2$. Then I can write the equation (\ref{modulAinftyEq}) as
\begin{multline}(d_{\mathcal{P}} +d_{V}) m_{n} -z B_{(n +1 ,n +2)} \xi _{(n +1 ,n +2)} m_{n +2} + \\
 -\frac{1}{2} \sum _{n_{1} +n_{2} =n}\frac{1}{n_{1} ! n_{2} !} \sum _{\sigma  \in \mathbb{S}_{n}}\sigma  (B_{(n_{1} +1 ,1)} o^{\mathcal{P}} (m_{n_{1} +1} \otimes m_{n_{2} +1})) =0. \label{modulAinfSigma}\end{multline}

\bigskip In the case when $V$ is a chain complexes with antisymmetric inner product of degree $l$ the $\mathcal{F}_{\mathcal{D}} \mathcal{P} -$algebra structures on $V$ is described again by $\widehat{m}_{n ,b}$ satisfying the equation (\ref{modAinfhat}). In this case $\mathcal{P}$ must be a $\mathcal{D} -$operad with $\mathcal{D} =\mathcal{K}^{ \otimes  -1 -l} \mathcal{L}$. If one rewrites this equation in terms of $\{m_{n}\}$ then one gets the equations (\ref{modulAinflambda}),
(\ref{modulAinfSigma}) with degree $(l -1)$ evaluations of the compositions
\begin{equation}\phi _{f ,f^{ \prime }}^{\mathcal{P}} =\mu _{G_{(I_{1} ,I_{2} ,b_{1} ,b_{2})}}^{\mathcal{P}} (e[1]^{ \otimes  -1 -l} \otimes (f[1] \wedge f^{ \prime } [1])) \label{phipffanti}
\end{equation}
\begin{equation}\xi _{f ,f^{ \prime }}^{\mathcal{P}} =\mu _{G_{n ,b}}^{\mathcal{P}} (e[1]^{ \otimes  -1 -l} \otimes (f[1] \wedge f^{ \prime } [1])) \label{dsetaPffanti}
\end{equation}and $B_{f ,f^{ \prime }} :V^{ \otimes \{f\}} \otimes V^{ \otimes \{f^{ \prime }\}} \rightarrow k [ -l]$ the degree $( -l)$ contraction with $B$. I have proved the following result:

\begin{proposition}
\label{propFdPstr}Let $\mathcal{P}$ be a twisted $\mathcal{K}^{ \otimes 1 -l} -$modular operad (respectively twisted $\mathcal{K}^{ \otimes  -1 -l} \mathcal{L} -$modular operad). The set $\{m_{n ,b}\}$, $m_{n ,b} \in (V^{ \otimes n} \otimes \mathcal{P} ((n ,b)))_{0}^{\mathbb{S}_{n}}$ defines a modular $\mathcal{F}_{\mathcal{D}} \mathcal{P} -$algebra structure on the chain complex $V$ with symmetric (respectively antisymmetric) inner product $B$ of degree $l$, $B :V^{ \otimes 2} \rightarrow k [ -l]$, iff $m_{n} =\sum _{b}z^{b} m_{n ,b}$ satisfies the equation (\ref{modulAinfSigma}). 
\end{proposition}

\section{Differential graded Lie algebra $\underset{n ,b}{\bigoplus }(V^{ \otimes n} \otimes \mathcal{P} ((n ,b)))^{\mathbb{S}_{n}} [1]$.\label{dgla}}
Recall that the equation defining the algebra structure over the \\
\ $c o b a r -$transformation of some cyclic operad $\mathcal{A}$ can be written as $[h ,h] =0$, where $h$ is a function on the symplectic affine $\mathcal{A} -$manifold. The $c o b a r -$transformation is the tree-level part of the Feynman transform. As we shall see below in Section \ref{freepalg}
the equation (\ref{modulAinflambda}) describing the algebra over the Feynman transform for the twisted modular
operad $\mathcal{P}$ is the principal equation of the Batalin-Vilkovisky $\mathcal{P} -$geometry on the affine $\mathcal{P}$-manifold. 

\subsection{Odd vector field on the space of morphisms $\mathbb{M}_{\mathcal{D}^{ \vee }} \mathcal{P}^{d u a l} \rightarrow \mathcal{E} [V]$.}
Let $V$ be chain complex with symmetric or antisymmetric inner product $B$ of degree $l$. Let $\mathcal{P}$ be a modular $\mathcal{D} -$operad, so that $\mathcal{D}^{ \vee }$ is the cocycle corresponding to the twisting of $\mathcal{E} [V]$, i.e. $\mathcal{D} =\mathcal{K}^{ \otimes 1 -l}$ for symmetric $B$ and $\mathcal{D} =\mathcal{K}^{ \otimes  -1 -l} \mathcal{L}$ for antisymmetric $B$. I explain in this subsection that the linear and quadratic terms in the equation (\ref{modulAinfSigma})
define the structure of differential graded Lie algebra on the graded $k$-vector space $\underset{n ,b}{\bigoplus }(V^{ \otimes n} \otimes \mathcal{P} ((n ,b)))^{\mathbb{S}_{n}} [1]$. 

Let $M o r(\mathbb{M}_{\mathcal{D}^{ \vee }} \mathcal{P}^{d u a l} ,\mathcal{E}[V])$ denotes the space of operad morphisms from $\mathbb{M}_{\mathcal{D}^{ \vee }} \mathcal{P}^{d u a l}$ to $\mathcal{E} [V]$. Since $\mathbb{M}_{\mathcal{D}^{ \vee }} \mathcal{P}^{d u a l}$ is a free modular operad I have \begin{equation*}M o r(\mathbb{M}_{\mathcal{D}^{ \vee }} \mathcal{P}^{d u a l} ,\mathcal{E}[V]) =(\underset{n ,b}{\bigoplus }(V^{ \otimes n} \otimes \mathcal{P} ((n ,b)))^{\mathbb{S}_{n}})_{0}\text{.}
\end{equation*} One can consider the corresponding graded version of the space of morphism  \\
\ $\underline{M o r}(\mathbb{M}_{\mathcal{D}^{ \vee }} \mathcal{P}^{d u a l} ,\mathcal{E}[V])$. It is the affine $\mathbb{Z} -$ graded scheme representing the functor $R \rightarrow M o r(\mathbb{M}_{\mathcal{D}^{ \vee }} \mathcal{P}^{d u a l} \otimes R ,\mathcal{E}[V])$ where $R$ is a graded commutatitve $k -$algebra. For the graded version I have \begin{equation*}\underline{M o r}(\mathbb{M}_{\mathcal{D}^{ \vee }} \mathcal{P}^{d u a l} ,\mathcal{E}[V]) =\underset{n ,b}{\bigoplus }(V^{ \otimes n} \otimes \mathcal{P} ((n ,b)))^{\mathbb{S}_{n}}\text{.}
\end{equation*} The differential $d_{\mathcal{F}}$ acting on $\mathbb{M}_{\mathcal{D}^{ \vee }} \mathcal{P}^{d u a l}$ induces the canonical odd vector field on $\underline{M o r}(\mathbb{M}_{\mathcal{D}^{ \vee }} \mathcal{P}^{d u a l} ,\mathcal{E}[V])$:
\begin{equation}Q (\varphi ) =d_{V} \varphi  -\varphi  d_{\mathcal{F}} . \label{qphi}
\end{equation}Since $d_{\mathcal{F}}^{2} =d_{V}^{2} =0$ it follows that
\begin{equation}[Q ,Q] =0. \label{qq}
\end{equation}The equation (\ref{dmmd}) describing the $\mathcal{F}_{\mathcal{D}} \mathcal{P} -$algebra structures on $V$ is precisely the equation \begin{equation*}Q (\varphi ) =0
\end{equation*} considered on the subspace $\deg \varphi  =0$. The same calculation as in the previous Section shows that the vector field $Q$ has only linear and quadratic components and they are given by the linear in $\{m_{n}\}$ and the quadratic in $\{m_{n}\}$ terms in (\ref{modulAinfSigma}).
The vector field $Q$ induces the odd coderivation of free cocommutative coalgebra generated by the $\mathbb{Z} -$ graded vector space $F =\underset{n ,b}{\bigoplus }(V^{ \otimes n} \otimes \mathcal{P} ((n ,b)))^{\mathbb{S}_{n}}$. The equation (\ref{qq}) is equivalent to three equations for the linear and quadratic components of
$Q$. These are the identities for a differential and a Lie bracket in a differential graded Lie algebra. I have proved the following
result.

\begin{proposition}
\label{propdgla}The linear and the quadratic terms in (\ref{modulAinfSigma})
are components of the structure of differential graded Lie algebra on $F [1]$ (in the category of chain complexes). In particular,
\begin{multline}\{m_{1} ,m_{2}\} =\frac{( -1)^{\bar{m}_{1} \bar{m}_{2}}}{n_{1} ! n_{2} !} \sum _{\sigma  \in \mathbb{S}_{n}}\sigma  (o^{\mathcal{E} [V]} \otimes o^{\mathcal{P}}) (m_{1} \otimes m_{2}) ,\; \label{brack} \\
m_{i} \in (V^{ \otimes n_{i} +1} \otimes \mathcal{P} ((n_{i} +1 ,b_{i})))^{\mathbb{S}_{n_{i} +1}} ,i =1 ,2\end{multline}defines the odd Lie bracket and
\begin{equation} \Delta m =( -1)^{\bar{m}} \xi _{n -1 ,n}^{\mathcal{E} [V]} \otimes \xi _{n -1 ,n}^{\mathcal{P}} m ,\text{\ \ }m \in (V^{ \otimes n} \otimes \mathcal{P} ((n ,b)))^{\mathbb{S}_{n}} \label{delta}
\end{equation}is a degree $( -1)$ differential. 
\end{proposition}

Notice that the bracket (\ref{brack})
is defined in terms of the compositions of the underlying cyclic operads, while the operator (\ref{delta}) is defined
in terms of the modular compositions along the graphs of type $G_{n ,b}$. Remark that the bracket essentually coincides in the case of cyclic operads with the Lie brackets described in \cite{KapMan}
and \cite{Ginz}. If I put $d m =(( -1)^{\bar{m} +1} d_{\mathcal{P}} -d_{V}) m$ then our basic equation (\ref{modulAinflambda}) becomes the familiar equation
\begin{equation}d m +z  \Delta m +\frac{1}{2} \{m ,m\} =0. \label{mster}
\end{equation}

\bigskip Combining Propositions \ref{propdgla}
and \ref{propFdPstr} I get the following result

\begin{theorem}
\label{theorem1}The modular $\mathcal{F}_{\mathcal{D}} \mathcal{P} -$algebra structures on the chain complex $V$ with symmetric (respectively antisymmetric) inner product $B$ of degree $l$, $B :V^{ \otimes 2} \rightarrow k [ -l]$, where $\mathcal{P}$ is an arbitrary twisted $\mathcal{K}^{ \otimes 1 -l} -$modular operad (respectively $\mathcal{K}^{ \otimes  -1 -l} \mathcal{L}$-moduar operad), are in one-to-one correspondence with solutions of the quantum master equation (\ref{mster})
in the space $(\underset{n ,b}{\bigoplus }(V^{ \otimes n} \otimes \mathcal{P} ((n ,b)))^{\mathbb{S}_{n}})_{0}$. 
\end{theorem}

This theorem is a generalizaion of the well-known results concerning algebras over $B a r -$transform of cyclic operads or co-operads, see \cite{kont},
\cite{GJ}. We shall see in the next Section that the space $(\underset{n ,b}{\bigoplus }(V^{ \otimes n} \otimes \mathcal{P} ((n ,b)))^{\mathbb{S}_{n}}$ can be interpreted as the vector space of hamiltonians, generating derivations preserving symplectic structure. Such description of
algebras over arbitrary Feynman transform via derivations preserving symplectic structure can also be compared with the characterisation of the Feynman
transform in the special case of modular completion of the commutative operad studied in \cite{M},
where an interpretation via \emph{higher} order coderivations of the free cocommutative coalgebra was given. 

If
one considers the modular $\mathcal{F}_{\mathcal{D}} \mathcal{P} -$algebra structures over some commutative graded algebra $C$ then they are in one-to-one correspondence with solutions to (\ref{mster}) in the space $(F \otimes C)_{0}$. One can define in the standard way using the algebra $C =k [\varepsilon ]/\varepsilon ^{2}$, the modular homotopy equivalence of the $\mathcal{F}_{\mathcal{D}} \mathcal{P} -$algebra structures. Then it is easy to see that the equivalence classes of the modular $\mathcal{F}_{\mathcal{D}} \mathcal{P} -$algebra structure are in one-to-one correspondence with the gauge equivalence classes of solutions to the quantum master equation (\ref{mster}).

\subsection{\bigskip DGLA of morphisms $\mathcal{F}_{\mathcal{D}} \mathcal{P} \rightarrow \tilde{\mathcal{P}}$.}
One may notice that the definitions of the odd bracket (\ref{brack}) and the odd differential (\ref{delta})
work in fact for an arbitrary pair $(\mathcal{P} ,\tilde{\mathcal{P}})$ where $\mathcal{P}$ is a modular $\mathcal{D} -$operad with arbitrary $\mathcal{D}$ and $\tilde{\mathcal{P}}$ is a modular $\mathcal{D}^{ \vee } -$operad. Then there is a natural differential graded Lie algebra structure on \begin{equation*}\underset{n ,b}{\bigoplus }(\tilde{\mathcal{P}} ((n ,b)) \otimes \mathcal{P} ((n ,b)))^{\mathbb{S}_{n}} [1]
\end{equation*} defined by the components of the vector field (\ref{qphi}). The solutions to the corresponding
Maurer-Cartan equation in $\underset{n ,b}{\bigoplus }(\tilde{\mathcal{P}} ((n ,b)) \otimes \mathcal{P} ((n ,b)))^{\mathbb{S}_{n}} [1]$ are in one-to one correspondence with morphisms of operads $\mathcal{F}_{\mathcal{D}} \mathcal{P} \rightarrow \tilde{\mathcal{P}}$. Notice that this equation coincides with the equation describing the morphisms $\mathcal{F}_{\mathcal{D}^{ \vee }} \tilde{\mathcal{P}} \rightarrow \mathcal{P}$. 

\section{Free $\mathcal{P} -$algebra.\label{freepalg}}
If $\mathcal{P}$ is a cyclic operad then the free $\mathcal{P} -$algebra generated by the graded $k -$vector space $V$ is \begin{equation*}C =\underset{n}{\bigoplus }(V^{ \otimes n} \otimes \mathcal{P} ((n +1)))^{\mathbb{S}_{n}}\text{.}
\end{equation*} It was argued in (\cite{kont},\cite{Ginz})
that the vector space \begin{equation*}F =\underset{n}{\bigoplus }(V^{ \otimes n} \otimes \mathcal{P} ((n)))^{\mathbb{S}_{n}}
\end{equation*} can be considered naturally as the analog of the space of functons on $S p e c (C)$. If $\mathcal{P}$ is a twisted modular operad then the compositions along trees form twisted version of cyclic operad. There
is the corresponding version of the free $\mathcal{P}$ algebra and the arguing can be repeated that $F$ can be seen as the space of functions on $S p e c (C)$ in the twisted case. 

Let $\mathcal{P}$ be a modular $D e t -$operad in the category of graded vector spaces and let $C y c \mathcal{P} = \oplus _{n}\mathcal{P} ((n ,0))$ is the cyclic operad, which is $b =0$ part of $\mathcal{P}$. Then the cobar transformation of $C y c \mathcal{P}$ is related to the Feynman transform as follows: the $b =0$ part of $\Sigma ^{ -1} \mathcal{F}_{D e t} \mathcal{P} =\mathcal{F}_{\mathcal{D}_{\Sigma } D e t} \Sigma  \mathcal{P}$ is equal to $\mathfrak{s} B C y c \mathcal{P}$, that is the suspension of the cobar transformation of $C y c \mathcal{P}$. In the framework of $\mathcal{Q} -$symplectic geometry associated with a cyclic operad $Q$, see \cite{kont} and also \cite{Ginz},
the $B \mathcal{Q} -$algebra structure on a vector space $V$ with symmetric inner product $\beta $ of degree zero is described by a function on affine $\mathcal{Q} -$manifold $\mathfrak{s} V$ of degree $\deg h =1$, such that
\begin{equation}[h ,h] =0 ,\text{\thinspace \thinspace \thinspace }h \in \underset{n}{\bigoplus }((\mathfrak{s} V)^{ \otimes n} \otimes \mathcal{Q} ((n)))^{\mathbb{S}_{n}} \label{clMast}
\end{equation}where the bracket is the Poisson bracket which is associated with the antisymmetric inner product $\mathfrak{s} \beta $. 

Let us put $\tilde{\mathcal{P}} =\Sigma  \mathcal{P}$ and consider the $\mathcal{F}_{\mathcal{D}} \tilde{\mathcal{P}} -$algebra stractures on $\mathfrak{s} V$, where $\mathcal{D} =\mathcal{D}_{\Sigma } D e t$. Then $b =0$ part of such structure is the same as the $B C y c \mathcal{P} -$algebra structure on $V$. The $\mathcal{F}_{\mathcal{D}} \tilde{\mathcal{P}} -$algebra structures on $\mathfrak{s} V$ is described, according to the Theorem from Section \ref{sectFeynm}, by an element $\widehat{h} (z) =\sum _{b \geq 0}h_{b} z^{b}$, $h_{b} \in \underset{n}{\bigoplus }((\mathfrak{s} V)^{ \otimes n} \otimes \mathcal{P} ((n ,b)))^{\mathbb{S}_{n}}$, $\deg  h_{b} =1$, such that\begin{equation*}z  \Delta \widehat{h} (z) +\frac{1}{2}[\widehat{h}(z) ,\widehat{h} (z)] =0.
\end{equation*} 

We see that in the ''classical'' limit $z \rightarrow 0$, $\widehat{h} (z)$ becomes a solution to (\ref{clMast}), the ''classical'' master equation, describing the $B C y c \mathcal{P} -$algebra structure on $V$. The operator $\Delta $ can be seen as the odd second order operator and the bracket is the odd Poisson bracket on $\underset{n ,b}{\bigoplus }((\mathfrak{s} V)^{ \otimes n} \otimes \Sigma  \mathcal{P} ((n ,b)))^{\mathbb{S}_{n}}$ extending the previous bracket from the subspace $b =0$. The whole picture is a noncommutative $\mathcal{P} -$analog of the usual commutative Batalin-Vilkovisky geometry described in \cite{S},\cite{W}.
It would be interesting to study the combinatorial consequences of $\mathcal{P} -$analogs of the Theorems of loc.cit. on invariance of integrals of $\exp  (\widehat{h} (z)/z)$ under deformations. 

\section{Characteristic classes of $\mathcal{F}_{\mathcal{D}} \mathcal{P} -$algebras.}
\bigskip Let $\widehat{m} :\mathcal{F}_{\mathcal{D}} \mathcal{P} \rightarrow $ $\mathcal{E} [V]$ be an $\mathcal{F}_{\mathcal{D}} \mathcal{P} -$algebra structure on the chain complex $V$ with symmetric or antisymmetric inner product $B$ of degree $l$. Here $\mathcal{P}$ is a modular $\mathcal{D} -$operad, such that $\mathcal{D}^{ \vee }$ is the cocycle corresponding to the twisting of $\mathcal{E} [V]$, i.e. $\mathcal{D} =\mathcal{K}^{ \otimes 1 -l}$ for symmetric $B$ and $\mathcal{D} =\mathcal{K}^{ \otimes  -1 -l} \mathcal{L}$ for antisymmetric $B$. It is one of the main application of the formalism developed in \cite{GK}
that the component of the morphism $\widehat{m}$ with $n =0$ \begin{equation*}\widehat{m} ((0 ,b)) :\underset{G \in [\Gamma ((0 ,b))] ,b >1}{\bigoplus }(\mathcal{D}^{ \vee } (G) \otimes \mathcal{P}^{d u a l} ((G)))_{A u t (G)} \rightarrow k
\end{equation*} is a cocycle on the subcomplex of $\mathcal{F}_{\mathcal{D}} \mathcal{P}$ corresponding to graphs with no external legs:
\begin{equation}\widehat{m} ((0 ,b))\vert _{\ensuremath{\operatorname*{Im}} d_{\mathcal{F}}} =0. \label{m0}
\end{equation}More generally if $\widehat{m}_{t} :\mathcal{F}_{\mathcal{D}} \mathcal{P} \rightarrow \mathcal{E} [V] \otimes k [t]$ is the $\mathcal{F}_{\mathcal{D}} \mathcal{P} -$operad structure depending on some graded parameters $t$ then all Taylor cofficient of expansion of the component $\widehat{m}_{t} ((0 ,b))$ at $t =0$ are cocycles on the subcomplex of graphs with no external legs
\begin{equation}\frac{ \partial ^{\vert \alpha \vert }}{ \partial t_{1}^{\alpha _{1}} \ldots   \partial t_{n}^{\alpha _{n}}}\widehat{m}_{t} ((0 ,b))\vert _{t =0 ,\ensuremath{\operatorname*{Im}} d_{\mathcal{F}}} =0 \label{dtm0}
\end{equation}It follows from (\ref{mhatG}) that the value $\widehat{m} ((0 ,b))$ on an element from \begin{equation*}(\mathcal{D}^{ \vee } (G) \otimes \mathcal{P}^{d u a l} ((G)))_{A u t (G)}
\end{equation*} corresponding to the stable graph $G$ is given by the partition function obtained by contracting the product \begin{equation*}\underset{v \in V e r t (G)}{\bigotimes }m_{L e g (v) ,b (v)}
\end{equation*} with $B^{ \otimes E d g e (G)}$. Similarly the cocycle (\ref{dtm0}) is a partition function involving insertions of the derivatives
$\frac{ \partial ^{\vert \beta \vert }}{ \partial t_{1}^{\beta _{1}} \ldots   \partial t_{n}^{\beta _{n}}}m_{L e g (v) ,b (v)}\vert _{t =0}$ with $\beta _{i} \leq \alpha _{i}$ at vertices of $G$ so that for all $1 \leq i \leq n$ the total sum of $\beta _{i}$ for all such insertions in the graph $G$ is equal to $\alpha _{i}$. 

\section{Stable ribbon graphs.}
In the Section below I shall use this construction in order to construct the homology classes in Deligne -Mumford moduli spaces associated with
solutions to the master equation (\ref{mster}) on affine $\mathbb{S} [t] -$manifold, where $\mathbb{S} [t]$ is the modular $D e t -$operad introduced below in Section \ref{detst}. I use the complex of stable ribbon graphs and its
relation with a compactification of moduli spaces of algebraic curves described in \cite{konts3},
see also \cite{Looijenga}. This is a generalization of the equivalence of
''decorated''\ moduli spaces of algebraic curves and moduli spaces of ribbon graphs due to J.Harer,
D.Mumford, R.C.Penner, W.Thurston and others. 

A stable ribbon graph is a connected graph $G$ together with :

\begin{itemize}
\item partitions of the set of flags adjacent to every vertex into $i (v)$ subsets \begin{equation*}L e g (v) =L e g (v)^{(1)} \sqcup \ldots  \sqcup L e g (v)^{(i (v))} ,v \in V e r t (G)
\end{equation*} 

\item fixed cyclic order on every subset $L e g (v)^{(k)}$ , 

\item a map $g :V e r t (G) \rightarrow \mathbb{Z}_{ \geq 0}$ such that for any vertex $2 (2 g (v) +i (v) -2) +n (v) >0$, so that putting $b (v) =2 g (v) +i (v) -1$ defines a stable graph. \end{itemize}

Let us denote via $S R_{(n ,b)}$ the set of all stable ribbon graphs with $n$ exterrior legs and $b (G) =b$. The usual ribbon graphs correspond to the case $b (v) =0$ for all $v \in V e r t (G)$. 

It is easy to see that for graphs from $S R_{(0 ,b)}$ our definition is equivalent to the definition given in \cite{konts3}
via the limit of a certain functor on ribbon graphs. A metric on the stable ribbon graph is a function $l :E d g e (G) \rightarrow \mathbb{R}_{ >0}$. Given a stable ribbon graph $G \in S R_{(0 ,b)}$ and a metric on $G$ one can construct by the standard procedure a punctured Riemann surface $\Sigma  (G)$, which will have double points in general. Namely one should replace every edge by oriented open strip $[0 ,l] \times ]  -i \infty  , +i \infty [$ and glue them for each cyclically ordered subset according to the cyclic order.
In this way one gets several punctured Riemann surfaces and for every vertex of the graph $G$ one should identify the points on these surfaces corresponding to different cyclically ordered subsets associated with the given
vertex of $G$. We also have the nonnegative integer $g (v)$ associated to every singular point of the Riemann surface $\Sigma  (G)$. The one-dimensional CW-complex $\vert G\vert $ is naturally realized as a subset of $\Sigma  (G)$. One can also construct in the similar way the Riemman surface associated with stable ribbon graph $G$ having legs. In such case one gets the singular Riemann surface associated with the graph $G/L e g s (G)$, i.e. the stable ribbon graph $G$ with legs removed, plus the extra structure, which consists of the lines on $\Sigma  (G/L e g s (G))\text{,}$one for each leg, which connect the vertex with the corresponding adjacent puncture, so that $\vert G\vert $ is again naturally realized as a subset of $\Sigma  (G)$. I shall denote in the sequel the set of punctures of the surface $\Sigma  (G)$ via $P_{\Sigma  (G)}$. 

One can consider the moduli space $\bar{\mathcal{M}}_{\gamma  ,\nu }^{c o m b}$ parametrizing the equivalence classes of data $(G ,l)$, where $G$ is a graph from $S R_{(0 ,b)}$ whose associated surface $\Sigma  (G)$ has genus $\gamma $ and exactly $\nu $ punctures numbered from $1$ to $\nu \text{,}$ and $l$ is a metric on $G$. It can be shown, see loc.cit. and \cite{Looijenga},
that there is a natural factor space $\bar{\mathcal{M}}_{\gamma  ,\nu }^{ \prime }$ of the Deligne-Mumford moduli space of stable curves $\bar{\mathcal{M}}_{\gamma  ,\nu }$ so that $\bar{\mathcal{M}}_{\gamma  ,\nu }^{c o m b}$ is homeomorphic to $\bar{\mathcal{M}}_{\gamma  ,\nu }^{ \prime } \times \mathbb{R}_{ >o}^{\nu }$ and the projection to $\mathbb{R}_{ >o}^{\nu }$ corresponds to the map which sends stable graph with metric and numbered punctures to the set of perimeters of edges surrounding
the punctures. In particular the preimage of $p =(p_{1\text{,}} \ldots  ,p_{\nu })$, $p \in $ $\mathbb{R}_{ >o}^{\nu }$ in $\bar{\mathcal{M}}_{\gamma  ,\nu }^{c o m b}$ can be considered as the moduli space $\bar{\mathcal{M}}_{\gamma  ,\nu }^{c o m b} (p)$ of data $(G ,l)$ such that the perimeters around punctures are equal to $p_{1\text{,}} \ldots  ,p_{\nu }$. This moduli space $\bar{\mathcal{M}}_{\gamma  ,\nu }^{c o m b} (p)$ is then homeomorphic to $\bar{\mathcal{M}}_{\gamma  ,\nu }^{ \prime }$. The space $\bar{\mathcal{M}}_{\gamma  ,\nu }^{c o m b} (p)$ has natural structure of a cell complex, or better say, orbi-cell complex, with (orbi-)cells indexed by equivalence classes of stable
graphs $G$ with numbered punctures as above. \\

\section{\bigskip Modular $D e t -$operad $\mathbb{S} [t]$.\label{detst}}
Let us introduce the following modular $D e t -$operad $\mathbb{S} [t]$. Let $k [\mathbb{S}_{n}]^{ \prime }$ denotes the graded $k -$vector space with the basis indexed by elements $(\sigma  ,a_{\sigma })$, where $\sigma  \in \mathbb{S}_{n}$ is a permutation with $i_{\sigma }$ cycles $\sigma _{\alpha }$ and $a_{\sigma } =\sigma _{1} \wedge \ldots  \wedge \sigma _{i_{\sigma }}$, $a_{\sigma } \in D e t (C y c l e (\sigma ))$ is one of the generators of the one-dimensional determinant of the set of cycles of $\sigma $, i.e. $a_{\sigma }$ is an order on the set of cycles defined up to even reordering, and $(\sigma  , -a_{\sigma }) = -(\sigma  ,a_{\sigma })$. The symmetric group $\mathbb{S}_{n}$ acts naturally on $k [\mathbb{S}_{n}]^{ \prime }$ by conjugation. Let $k [t]$ denotes the space of polynomials in one variable $t$, $\deg t = -2$. The cyclic $\mathbb{S} -$module underlying our modular operad is the set of graded $\mathbb{S}_{n} -$modules \begin{equation*}\mathbb{S} [t] ((n)) =k [\mathbb{S}_{n}]^{ \prime } [ -1] \otimes k [t]
\end{equation*} where $\mathbb{S}_{n}$ acts trivially on $k [t]$, and the degree $b$ of the element $(\sigma  ,a_{\sigma })t^{g}$ is defined by \begin{equation*}b =2 g +i_{\sigma } -1.
\end{equation*} In particular $\mathbb{S} [t] ((n ,b))$ is a graded $k -$vector space concentrated in degree $( -b)$. I also put $\mathbb{S} [t] ((n ,0)) =0$ for $n \leq 2$. Notice that for $b =0$ I get the underlying $\mathbb{S} -$module of the cyclic operad $A s s$ of associative algebras with invariant scalar products. 

Compositions in $\mathbb{S} [t]$ are $k [t] -$linear and defined via sewings and dissections of cycles of permutations. The compositions can be easily described using multiplication
on the group of permutations. Let us describe the composition $\mu _{G_{(I ,J ,b ,b^{ \prime })}}^{\mathbb{S} [t]}$ along the simplest graph with two vertices (\ref{phiij}). Let $(\sigma  ,a_{\sigma })t^{g} \in \mathbb{S} [t] ((I \sqcup \{f\} ,b))$, $a_{\sigma } =\sigma _{1} \wedge \ldots  \wedge \sigma _{i_{\sigma }}$, $(\rho  ,a_{\rho })t^{g^{ \prime }} \in \mathbb{S} [t] ((J \sqcup \{f^{ \prime }\} ,b^{ \prime }))$, $a_{\rho } =\rho _{1} \wedge \ldots  \wedge \rho _{i_{\rho }}$ with $f$ belonging to the cycle $\sigma _{k}$ and $f^{ \prime }$ belonging to the cycle $\rho _{l}$. Let us denote by $\pi _{f ,f^{ \prime }}$ \begin{equation*}\pi _{f ,f^{ \prime }} :A u t (\{1 ,\ldots  ,n\} \sqcup \{f ,f^{ \prime }\}) \rightarrow \mathbb{S}_{n}
\end{equation*} the operation erasing the elements $f$ and $f^{ \prime }$ from the cycles of permutation \begin{equation*}\pi _{f ,f^{ \prime }} :(i_{1} \ldots  i_{\alpha } f j_{1} \ldots  j_{\beta } f^{ \prime }) \rightarrow (i_{1} \ldots  i_{\alpha } j_{1} \ldots  j_{\beta })\text{.}
\end{equation*} If $b =b^{ \prime } =0$ then I have simply the cyclic permutations and the composition $\mu _{G_{(I ,J ,b ,b^{ \prime })}}^{\mathbb{S} [t]}$coincides with the composition in the cyclic operad $A s s$, which can be written as $\pi _{f ,f^{ \prime }} \sigma  \rho  (f f^{ \prime })$ where $(f f^{ \prime })$ is the transposition $f \leftrightarrow f^{ \prime }$. For general elements of $\mathbb{S} [t]$ I have the following expression \begin{equation*}\mu _{G_{(I ,J ,b ,b^{ \prime })}}^{\mathbb{S} [t]} =(\pi _{f ,f^{ \prime }} \sigma  \rho  (f f^{ \prime }) ,a_{\mu })t^{g +g^{ \prime }}
\end{equation*} where \begin{equation*}a_{\mu } =( -1)^{k +l} \pi _{f ,f^{ \prime }} (\sigma _{k} \rho _{l} (f f^{ \prime })) \wedge \sigma _{1} \wedge \ldots  \wedge \widehat{\sigma }_{k} \wedge \ldots  \wedge \sigma _{i_{\sigma }} \wedge \rho _{1} \wedge \ldots  \wedge \widehat{\rho }_{l} \wedge \ldots  \wedge \rho _{i_{\rho }}\text{.}
\end{equation*} I leave to an interested reader to verify that the sign in the expression for $a_{\mu }$ follows from the natural isomorphisms
\begin{multline}D e t(C y c l e (\pi _{f ,f^{ \prime }} \sigma  \rho  (f f^{ \prime })) \simeq D e t (\{\pi _{f ,f^{ \prime }} \sigma _{k} \rho _{l} (f f^{ \prime })\}) \otimes  \\
 \otimes D e t (C y c l e (\sigma ) \setminus \{\sigma _{k}\} \sqcup C y c l e (\rho ) \setminus \{\rho _{l}\}) , \\
D e t (C y c l e (\sigma )) [ -1] \simeq D e t (C y c l e (\sigma ) \setminus \{\sigma _{k}\}) , \label{Detcyclepiff} \\
D e t (C y c l e (\rho )) [ -1] \simeq D e t (C y c l e (\rho ) \setminus \{\rho _{l}\}) .\end{multline}

Let us describe the composition along the graph with one loop (\ref{dzetaij})\begin{equation*}\mu _{G_{n ,b}}^{\mathbb{S} [t]} :D e t (H_{1} (G_{n ,b})) \otimes \mathbb{S} [t] ((\{1 ,\ldots  ,n\} \sqcup \{f ,f^{ \prime }\} ,b -1)) \rightarrow \mathbb{S} [t] ((n ,b))\text{.}
\end{equation*} Let $(\sigma  ,a_{\sigma })t^{g} \in \mathbb{S} [t] ((\{1 ,\ldots  ,n\} \sqcup \{f ,f^{ \prime }\} ,b -1))$, $a_{\sigma } =\sigma _{1} \wedge \ldots  \wedge \sigma _{i_{\sigma }}$ and $f \in \sigma _{k}$ and $f^{ \prime } \in \sigma _{l}$ with $k <l$. The pair $f ,f^{ \prime }$ defines the oriented loop $e_{f f^{ \prime }}$ and hence an element $e_{f f^{ \prime }} [1]$ of $D e t (H_{1} (G_{n ,b}))$. Then the composition $\mu _{G_{n ,b}}^{\mathbb{S} [t]}$ on $e_{f f^{ \prime }} [1] \otimes (\sigma  ,a_{\sigma }) t^{g}$ is the sewing of cycles $\sigma _{k}$ and $\sigma _{l}$ times $t$ \begin{equation*}\mu _{G_{n ,b}}^{\mathbb{S} [t]} =(\pi _{f ,f^{ \prime }} \sigma  (f f^{ \prime }) ,a_{\mu ^{ \prime }})t^{g +1}
\end{equation*} where \begin{equation*}a_{\mu ^{ \prime }} =( -1)^{k +l -1} \pi _{f ,f^{ \prime }} (\sigma _{k} \sigma _{l} (f f^{ \prime })) \wedge \sigma _{1} \wedge \ldots  \wedge \widehat{\sigma }_{k} \wedge \ldots  \wedge \widehat{\sigma }_{l} \wedge \ldots  \wedge \sigma _{i_{\sigma }}
\end{equation*} following a sequence of natural isomorphisms analogous to (\ref{Detcyclepiff}).
If the elements $f$ and $f^{ \prime }$ belong to the same cycle $f ,f^{ \prime } \in \sigma _{k}$ then the value of the composition $\mu _{G_{n ,b}}^{\mathbb{S} [t]}$ on $e_{f f^{ \prime }} [1] \otimes (\sigma  ,a_{\sigma }) t^{g}$ is the dissection of the cycle $\sigma _{k}$ into two cycles whose relative order in $D e t (C y c l e)$ is determined by the orientation of the edge $e_{f f^{ \prime }}$: \begin{equation*}\mu _{G_{n ,b}}^{\mathbb{S} [t]} =(\pi _{f ,f^{ \prime }} \sigma  (f f^{ \prime }) ,a_{\mu ^{ \prime  \prime }})t^{g}
\end{equation*} where, if I denote by $\sigma _{k}^{f}$ and $\sigma _{k}^{f^{ \prime }}$ the two cycles of $\sigma _{k} (f f^{ \prime })$ containing $f$ and $f^{ \prime }$ correspondingly, then we have \begin{equation*}a_{\mu ^{ \prime  \prime }} =( -1)^{k -1} (\pi _{f} \sigma _{k}^{f}) \wedge (\pi _{f^{ \prime }} \sigma _{k}^{f^{ \prime }}) \wedge \sigma _{1} \wedge \ldots  \wedge \widehat{\sigma }_{k} \wedge \ldots  \wedge \sigma _{i_{\sigma }}
\end{equation*} which follows from natural isomorphisms analogous to (\ref{Detcyclepiff}). Remark
that if one of the cycles $\sigma _{k}^{f}$, $\sigma _{k}^{f^{ \prime }}$ consist of just one element $f$ or $f^{ \prime }$ correspondingly, which happens precisely when $f$ and $f^{ \prime }$are neighbours in the cycle $\sigma _{k}$, then the composition is zero in such case ($(\pi _{f} \sigma _{k}^{f}) =0$ or $(\pi _{f^{ \prime }} \sigma _{k}^{f^{ \prime }}) =0$). 

One can check that these compositions define on $\mathbb{S} [t]$ the structure of twisted modular $D e t -$operad. Namely, contraction of several edges corresponds, forgetting the elements from $k [t]$ and $D e t (C y c l e)$, to successive operators of multiplications by transpositions followed by erasing operators. But operators corresponding to different
edges commute $[\pi _{f ,f^{ \prime }} ,\pi _{g g^{ \prime }}] =0$, $[\pi _{f ,f^{ \prime }} ,(g g^{ \prime })] =0$. It follows that the composition on the level of permutations is associative with respect to the morphisms of stable graphs. It follows
from commutativity of the diagrams of natural isomorphisms that the rules for compositions of decorations from $k [t]$ and $D e t (C y c l e)$ are also compatible with morphisms of stable graphs. 

\section{Feynman transform of $\mathbb{S} [t]$ and stable ribbon graphs.\label{sectionfeynmnst}}
Let us consider the Feynman transform of $\mathbb{S} [t]$. Notice that $\mathbb{S} [t] ((n))$ has a basis labeled by partitions of $[n]$ into $i$ subsets with cyclic orders on the subsets, plus the nonnegative integer $g$, such that $2 (2 g +i -2) +n >0$ and plus a choice of the ordering of cycles from $D e t (C y c l e)$. It follows immediately from the definition that the Feynman transform $\mathcal{F}_{D e t} \mathbb{S} [t]$ has the basis labeled by the pairs $(G ,\alpha _{G})$, where $G$ is a stable ribbon graph and $\alpha _{G}$ is a choice of orientation from
\begin{equation}\mathcal{K} D e t^{ -1} (G)( \otimes _{v \in V e r t (G)}D e t^{ -1} (C y c l e)[1 -2 g(v)]) . \label{KDetDet}
\end{equation}If $G$ is a stable ribbon graph let $\nu  (G)$ be the number of punctures of the Riemann surface $\Sigma  (G)$ associated with $G$. Then $(G ,\alpha _{G})$ belongs to the subspace $\mathcal{F}_{D e t} \mathbb{S} [t] ((n ,b))$ where $n$ is the number of legs of $G$ and \begin{equation*}b (G) =2 \gamma  (G) -1 +\nu  (G)
\end{equation*} where $\gamma  (G)$ is the genus of $\Sigma  (G)$ taking into account the genus defects associated with vertices of $G$ \begin{equation*}\gamma  (G) =g (N \Sigma  (G)) +\sum _{v \in V e r t (G)}(g (v) +i (v) -1)
\end{equation*} where $g (N \Sigma  (G))$ is the genus of the normalization of $\Sigma  (G)$. It is easy to see that every chain complex $\mathcal{F}_{D e t} \mathbb{S} [t] ((n ,b))$ is in fact the direct sum of complexes\begin{equation*}\mathcal{F}_{D e t} \mathbb{S} [t] ((n ,b)) =\underset{b =2 \gamma  -1 +\nu }{\bigoplus }\mathcal{F}_{D e t} \mathbb{S} [t] ((n ,\gamma  ,\nu )\text{.}
\end{equation*} 

The moduli space $\bar{\mathcal{M}}_{\gamma  ,\nu }^{c o m b} (p)/\mathbb{S}_{\nu }$ has natural (orbi-)cellular decomposition with (orbi-)cells indexed by the isomorphism classes of stable ribbon graphs with $\gamma  =\gamma  (G)$, $\nu  =\nu  (G)$ and $\vert L e g (G)\vert  =0$. Let us show that this cochain cell complex is identified naturally with the complex $\mathcal{F}_{D e t} \mathbb{S} [t] ((0 ,\gamma  ,\nu ))$.

\begin{theorem}
\label{theorem2}$H^{i} (\bar{\mathcal{M}}_{\gamma  ,\nu }^{ \prime }/\mathbb{S}_{\nu }) \simeq H_{ -i}(\mathcal{F}_{D e t} \mathbb{S}[t]((0 ,\gamma  ,\nu )))[2 \gamma  -1]\text{.}$ 
\end{theorem}

\begin{proof}
It is easy to see from the definition of the moduli space $\bar{\mathcal{M}}_{\gamma  ,\nu }^{c o m b} (p)$ that the (orbi-)cells, lying in the boundary of the orbi-cell of the equivalence class of stable ribbon graph $G$, correspond to the equivalence classes of the stable ribbon graphs of the form $G/\{e\}$, $e \in E d g e (G)$, see \cite{Looijenga}, \cite{Z}.
It remains to check that the coorientaition of the orbi-cell is naturally identified with an element from (\ref{KDetDet}),
so that the differential on $\mathcal{F}_{D e t} \mathbb{S} [t] ((0 ,\gamma  ,\nu ))$ coincides with the differential of the cochain cell complex of $\bar{\mathcal{M}}_{\gamma  ,\nu }^{ \prime }/\mathbb{S}_{\nu }$. This identification is analogous to the Proposition 9.5 of \cite{GK}.
Firstly, for every stable ribbon graph $G$ the vector space $\mathcal{K} (G)$ is naturally isomorphic to $H_{c}^{t o p} (C_{G})$ where $C_{G}$ is the (orbi-)cell corresponding to $G$. To prove that the complex $\mathcal{F}_{D e t} \mathbb{S} [t] ((0 ,\gamma  ,\nu )) [1 -2 \gamma ]$ computes indeed the cohomology $ \oplus _{i}H^{i} (\bar{\mathcal{M}}_{\gamma  ,\nu }^{ \prime }/\mathbb{S}_{\nu })$ it is sufficient to identify the inverse to the sheaf of cohomology of the fibers of projection $\bar{\mathcal{M}}_{\gamma  ,\nu }^{c o m b}/\mathbb{S}_{\nu } \rightarrow \bar{\mathcal{M}}_{\gamma  ,\nu }^{ \prime }/\mathbb{S}_{\nu }$, i.e it is sufficient to show that
\begin{equation}D e t (G)^{ -1} \underset{v \in V e r t (G)}{\bigotimes }D e t^{ -1}(C y c l e (\sigma  (v))[1 -2 g(v)] \simeq D e t^{ -1} (P_{\Sigma  (G)}) [1 -2 \gamma ] . \label{DetSigmatilda}
\end{equation}In order to prove this let us consider the surface $\tilde{\Sigma } (G)$ which is obtained topologically as follows. Let us remove a small neighborhoud of every singular point $v$ of $\Sigma  (G)$. If $v$ has the genus defect $g (v)$ and there are $i (v)$ branches that are meeting at $v$, then let us glue instead of this neighborhoud a curve $\Sigma  (v)$ of the genus $g (v)$ and with $i (v)$ boundary components. I obtain topologically a curve of genus $\gamma $ without $\nu $ marked points. Then
\begin{equation}D e t (H^{ \ast } (\tilde{\Sigma } (G) , \sqcup _{v \in V e r t (G)}\Sigma  (v))) \simeq D e t(H^{ \ast } (\vert G\vert  ,V e r t (G)) \simeq D e t (G)^{ -1}[1 -\vert  V e r t(G)\vert ] \label{detssigma}
\end{equation}since $(\tilde{\Sigma } (G) , \sqcup _{v \in V e r t (G)}\Sigma  (v))$ is homotopic to $(\vert G\vert  ,V e r t (G))$. Using the Poincare duality for the compact surface $\tilde{\Sigma } (G) \sqcup P_{\Sigma  (G)}$, where recall that $P_{\Sigma  (G)}$ denotes the set of marked points, and the Mayer-Vietoris sequence as in loc.cit. I get \begin{equation*}D e t (H^{ \ast } (\tilde{\Sigma } (G))) \simeq D e t^{ -1} (P_{\Sigma  (G)}) [2 -2 \gamma ]\text{.}
\end{equation*} Also, by similar arguments \begin{equation*}D e t (H^{ \ast } (\Sigma  (v))) \simeq D e t^{ -1}(C y c l e (\sigma  (v))[2 -2 g(v)]\text{.}
\end{equation*} Now the equality (\ref{DetSigmatilda}) follows from (\ref{detssigma})
and the exact sequence associated with the pair $(\tilde{\Sigma } (G) , \sqcup _{v \in V e r t (G)}\Sigma  (v))$. 
\end{proof}

\section{Algebras over Feynman transform of $\mathbb{S} [t]$.}

\subsection{\bigskip Coboundary $\mathcal{D}_{\protect \chi _{l}}$.}
The operad $\mathcal{F}_{D e t} \mathbb{S} [t]$ is a twisted modular $\mathcal{K} D e t^{ -1} -$operad. To consider algebras over $\mathcal{F}_{D e t} \mathbb{S} [t]$ I should apply a coboundary so that the twisting becomes compatible with corresponding twisting of $\mathcal{E} [V]$. Notice that \begin{equation*}\mathcal{K} D e t^{ -1} \simeq \mathcal{K}^{2 l} \mathcal{D}_{\chi _{l}}
\end{equation*} where $\chi _{l} =\tilde{\mathfrak{s}} \Sigma  \beta ^{2 (1 -l)}$. Let us put $\mathbb{S}_{\chi _{l}} [t] =\chi _{l} \mathbb{S} [t]$, so that $\mathbb{S}_{\chi _{l}} [t]$ is a twisted $\mathcal{K}^{1 -2 l}$-operad. This shifts the grading on our operad and also the $\mathbb{S}_{n} -$action is changed by the multiplication by the alternating representation: \begin{equation*}\mathbb{S}_{\chi _{l}} [t] ((n ,b)) =\mathbb{S} [t] ((n ,b)) \otimes s g n_{n}[n +( 2 l -2)(1 -b) -1]\text{.}
\end{equation*} Then the Feynman transform \begin{equation*}\mathcal{F}_{\mathcal{K}^{1 -2 l}} \mathbb{S}_{\chi _{l}} [t] =\chi _{l}^{ -1} \mathcal{F}_{D e t} \mathbb{S} [t]
\end{equation*} is the twisted modular $\mathcal{K}^{2 l}$-operad. In particular the twisting of the Feynman transform $\mathcal{F}_{\mathcal{K}} \mathbb{S}_{\chi _{0}} [t]$ is trivial and $\mathcal{F}_{\mathcal{K}} \mathbb{S}_{\chi _{0}} [t]$ is a usual modular operad. As the complex of vector spaces the Feynman transform $\mathcal{F}_{\mathcal{K}^{1 -2 l}} \mathbb{S}_{\chi _{l}} [t]$ is isomorphic to $\mathcal{F}_{D e t} \mathbb{S} [t]$ up to a shift of degrees \begin{equation*}\mathcal{F}_{\mathcal{K}^{1 -2 l}} \mathbb{S}_{\chi _{l}} [t] ((n ,b)) \simeq \mathcal{F}_{D e t} \mathbb{S} [t] ((n ,b))[1 -n +( 2 l -2)(b -1)]
\end{equation*} so that the characteristic classes of $\mathcal{F}_{\mathcal{K}^{1 -2 l}} \mathbb{S}_{\chi _{l}} [t] -$algebras take the value also in $ \oplus _{i}H_{i} (\bar{\mathcal{M}}_{\gamma  ,\nu }^{ \prime }/\mathbb{S}_{\nu })$. By the Theorem \ref{theorem1} from the section \ref{sectFeynm}
$\mathcal{F}_{\mathcal{K}^{1 -2 l}} \mathbb{S}_{\chi _{l}} [t] -$algebra on vector space $V$ with inner product $B$, is defined by the set of elements \begin{equation*}m_{n} \in \underset{n}{\bigoplus }(V^{ \otimes n} \otimes \mathbb{S}_{\chi _{l}}[t] ((n)))_{0}^{\mathbb{S}_{n}}
\end{equation*} satisfying the quantum master equation (\ref{mster}). If $V$ has a basis $\{e_{\alpha }\}_{\alpha  \in \Xi }$ then I can write
\begin{gather*}m_{n} =\sum _{I \subset \Xi  ,\vert I\vert  =n}\sum _{\substack{\sigma  \in A u t (I) ,g \\ 4 (g -1) +2 i_{\sigma } +n >0}}m_{\sigma  ,g} t^{g}\text{\ \ } \\
m_{\sigma  ,g} \in  \wedge _{i =1}^{i =i_{\sigma }}(V^{ \otimes \lambda _{i}})^{\mathbb{Z}/\lambda _{i} \mathbb{Z}}\; ,\;\deg  m_{\sigma  ,g} =(2 l -3) (2 -2 g -i_{\sigma }) +n\end{gather*}where $\sigma $ has $i_{\sigma }$ cycles of lenghts $\lambda _{1} ,\ldots  ,\lambda _{i_{\sigma }}$.

\begin{proposition}
The $\mathcal{F}_{\mathcal{K}^{1 -2 l}} \mathbb{S}_{\chi _{l}} [t] -$algebra structure on $V$ is defined by the set of tensors $\{m_{\sigma  ,g}\}$ satisfying the quantum master equation for functions on the affine $\mathbb{S}_{\chi _{l}} [t] -$manifold $V$. The partition function constructed from the tensors $m_{\sigma  ,g}$ defines the characteristic class with value in $ \oplus _{i}H_{i} (\bar{\mathcal{M}}_{\gamma  ,\nu }^{ \prime }/\mathbb{S}_{\nu })$. 
\end{proposition}

Notice that for an $\mathcal{F}_{\mathcal{K}^{1 -2 l}} \mathbb{S}_{\chi _{l}} [t] -$algebra the elements $m_{\sigma  ,g}$ with $i_{\sigma } =1$, $g =0$ define an $A_{\infty } -$algebra with inner product of degree $2 l$. The characteristic class map for the $A_{\infty } -$algebras was described in (\cite{Konts2}).
It takes values in the homologies of the uncompactified moduli spaces of curves. We see that the $\mathcal{F}_{\mathcal{K}^{1 -2 l}} \mathbb{S}_{\chi _{l}} [t] -$algebra is the algebraic structure whose operadic characteristic class map extends the $A_{\infty } -$algebra's map to the homologies of the compactification $\bar{\mathcal{M}}_{\gamma  ,\nu }^{ \prime }$. 

\subsection{\bigskip \ Modular $\mathcal{K}^{2 l} -$operads $\mathcal{F}_{\mathcal{K}^{1 -2 l}} \mathbb{S}_{\protect \chi _{l}} [t]$ and higher genus GW-invariants.}
The twisted modular $\mathcal{K}^{2 l} -$operad $\mathcal{F}_{\mathcal{K}^{1 -2 l}} \mathbb{S}_{\chi _{l}} [t]$ arises naturally in the counting of holomorphic curves of arbitrary genus with boundaries in the set of Lagrangian submanifolds of
a simplectic manifold of dimension $4 l$. Namely, let us consider a set of graded Lagrangian submanifolds, intersecting transversally, and let $\sigma $ be a permutation acting on the set of intersection points of these submanifolds, which I also mark by $\{1 ,\ldots  ,n\}$. Consider the total number of holomorphic maps of the surfaces of genus $g$ with $i_{\sigma }$ boundary components along with set of points lying on these boundaries, such that the components of the boundaries are mapped to the
Lagrangian submanifolds, and the points are mapped to the intersection points, so that the cyclic orders on the intersection points corresponding to every
boundary component form the permutation $\sigma $. These numbers define naturally the set of elements $\{m_{\sigma  ,g}\}$ from $\underset{n}{\bigoplus }(V^{ \otimes n} \otimes \mathbb{S}_{\chi _{l}}[t] ((n ,2 g +i -1)))^{\mathbb{S}_{n}}$ where $V$ is the graded vector spaces with the basis labeled by points of intersection of Lagrangian submanifolds. Then the standard arguments,
using the degeneration of one-parameter families of such maps, show that this set of elements satisfy the quantum master equation (\ref{mster})
of the noncommutative Batalin-Vilkovisky geometry on the affine $\mathbb{S}_{\chi _{l}} [t] -$manifold. Similarly, there exists the odd version, corresponding to the case of symplectic manifold of dimension $4 l +2$. It corresponds to the modular operad of $\mathbb{S}_{n} -$modules $\tilde{\mathbb{S}} [t] ((n)) =k [\mathbb{S}_{n}] \otimes k [t]$ with compositions defined as for $\mathbb{S} [t]$, by simply omitting the terms involvinvg $D e t (C y c l e (\sigma ))$. Then the complex underlying the Feynman transform $\mathcal{F} \tilde{\mathbb{S}} [t]$ is identified with cochain complex of $\bar{\mathcal{M}}_{\gamma  ,\nu }^{ \prime }/\mathbb{S}_{\nu }$ with coefficients in the local system $D e t (P_{\Sigma })$ . The counting of holomorphic curves with boundaries in Lagrangian submanifolds of a symplectic manifold of dimension $4 l +2$ defines an algebra over $\mathcal{F}_{\mathcal{K}^{ -2 l}} \beta ^{ -2 l} \tilde{\mathbb{S}} [t]$. One can show that this leads to a combinatorial description of Gromov-Witten invariants via the characteristic class map and the periodic
cyclic homology of the twisted modular operads along the lines of \cite{Bar1},\cite{BK}.

\subsection{\bigskip The map $\mathcal{F} \mathcal{A} s s \rightarrow \mathbb{S}_{\protect \chi _{0}} [t]$}
There is a close interplay between the twisted modular operad $\mathbb{S} [t]$ and the Feynman transform of the cyclic operad $\mathcal{A} s s$ which I would like to illustrate in this subsection. Forgetting the differential the Feynman transform $\mathcal{F} \mathcal{A} s s$ of the cyclic operad $\mathcal{A} s s$ is a $\mathcal{K} -$operad generated by the $\mathbb{S} -$module $\mathcal{A} s s^{d u a l}$. It follows that $\mathcal{F} \mathcal{A} s s$ has a basis labeled by ribbon graphs with a choice of a generator of the one-dimensional vector space $\mathcal{K} (G)$. The complex $\mathcal{F} \mathcal{A} s s ((n ,b))$ is decomposed as the sum of subcomplexes $\mathcal{F} \mathcal{A} s s ((n ,\gamma  ,\nu ))$ according to the genus $\gamma $ and the number of punctures $\nu $ of the Riemann surface associated with the ribbon graph, see \cite{GK},
proposition 9.2. Recall that $D e t (G) \simeq \mathcal{K} \mathcal{D}_{\chi _{0}}^{ -1}$ where $\mathcal{D}_{\chi _{0}}$ is the coboundary associated with the $\mathbb{S} -$module $\chi _{0} =\mathfrak{s} \Sigma $. It follows that $\chi _{0}^{ -1} \mathcal{F} \mathcal{A} s s$ is a modular $D e t -$operad which as a $k -$vector space consist of linear combinations of elements $(G ,\widehat{\alpha }_{G})$ where $G$ is a ribbon graph and $\widehat{\alpha }_{G}$ is an element of the one-dimensional vector space $D e t (G) \otimes (  \otimes _{v \in V e r t (G)}D e t(L e g (v)[ -3])$. The subset of legs of ribbon graph $G$ adjacent to a given puncture has natural cyclic order. It follows that every ribbon graph $G$ defines naturally a permutation $\sigma _{G}$ on the set $L e g (G)$. Notice that for graphs with at least one leg adjacent to every puncture I have
\begin{equation}D e t (G) \simeq D e t (C y c l e (\sigma _{G})) [2 \gamma  -1] \label{alphaG}
\end{equation}see loc.cit, page 117. Let $G$ be a trivalent ribbon graph. The cyclic order on $L e g (v)$ gives a canonical element in $D e t (L e g (v)) [ -3]$ for every vertex $v$. This is the element $e_{1} \wedge e_{2} \wedge e_{3}$ if $e_{1} \rightarrow e_{2} \rightarrow e_{3} \rightarrow e_{1}$ denotes the cyclic order on the three flags. Let $\widehat{\alpha }_{G}^{c a n}$ denotes the product of an element $\alpha _{G} \in D e t (G)$ with the tensor product of the canonical elements in $ \otimes _{v \in V e r t (G)}D e t (L e g (v)) [ -3]$. Let $\alpha _{\sigma _{G}}$ denotes the element from $D e t (C y c l e (\sigma _{G})) [2 \gamma  -1]$ corresponding to $\alpha _{G}$ under the isomorphism (\ref{alphaG}). I state here the following result, the proof is a simple
check.

\begin{proposition}
Let $G$ be a trivalent ribbon graph such that for every puncture of $G$ there is a leg of $G$ adjacent to this puncture. Let us put $\phi  (G ,\widehat{\alpha }_{G}^{c a n}) =$ $(\sigma _{G} ,\alpha _{\sigma _{G}})t^{\gamma }$ for such graph and $\phi  (G ,\widehat{\alpha }_{G}) =0$ for all other ribbon graphs. Then $\phi $ defines a morphism of twisted modular $D e t -$operads $\chi _{0}^{ -1} \mathcal{F} \mathcal{A} s s \rightarrow \mathbb{S} [t]$. 
\end{proposition}

\end{document}